\documentclass[12pt]{amsart}
\usepackage{amsmath,amssymb,amsthm,mathrsfs}
\usepackage[a4paper,margin=29mm]{geometry}
\usepackage{microtype,booktabs,tikz,graphicx}
\usepackage{comment}
\usepackage{enumerate}
\usepackage[colorlinks=true,linkcolor=blue,citecolor=blue,urlcolor=blue]{hyperref}
\hypersetup{pdftitle={Fourier Restriction Beyond Dimension Data: A Multiscale Moment Criterion},pdfauthor={Jinjun Li}}
\allowdisplaybreaks[2]
\numberwithin{equation}{section}
\newtheorem{theorem}{Theorem}[section]
\newtheorem{proposition}[theorem]{Proposition}
\newtheorem{lemma}[theorem]{Lemma}

\newtheorem{corollary}[theorem]{Corollary}
\newtheorem*{maintheoremA}{Main Theorem A}
\newtheorem*{maintheoremB}{Main Theorem B}
\newtheorem*{maintheoremC}{Main Theorem C}
\theoremstyle{definition}

\theoremstyle{remark}

\newtheorem{remark}{Remark}[section]
\newtheorem{example}{Example}[section]
\newcommand{\R}{\mathbb R}

\newcommand{\Z}{\mathbb Z}
\newcommand{\F}{\mathbb F}
\newcommand{\E}{\mathbb E}
\newcommand{\Prob}{\mathbb P}
\newcommand{\dd}{\,\mathrm d}
\newcommand{\one}{\mathbf 1}
\newcommand{\dimF}{\dim_{\mathrm F}}
\newcommand{\dimH}{\dim_{\mathrm H}}
\newcommand{\dimtwo}{\dim_2}
\newcommand{\dimS}{\dim_{\mathrm S}}
\newcommand{\supp}{\operatorname{supp}}
\newcommand{\dist}{\operatorname{dist}}
\newcommand{\Corr}{\mathcal C}
\newcommand{\unif}{\lambda_0}
\newcommand{\maskerr}{\varepsilon}
\newcommand{\bigconv}{\mathop{\vcenter{\hbox{\scalebox{1.8}{$\displaystyle\ast$}}}}\displaylimits}
\newcommand{\conv}{\mathbin{\vcenter{\hbox{\scalebox{1.2}{$\ast$}}}}}

\title[Fourier Restriction Beyond Dimension Data]
{Fourier Restriction Beyond Dimension Data: A Multiscale Moment Criterion}
\author{Jinjun Li}

\address[Jinjun Li]{School of Mathematics and Information Science, Guangzhou University, Guangzhou, 510006, P.~R.~China}
\email{lijinjun@gzhu.edu.cn}

\date{}

\begin{document}
\begin{abstract}
We establish an exact restriction criterion for weighted Moran measures
with uniformly bounded one-level frame bounds: for every $2<s<\infty$,
$L^2(\mu)\to L^s(\mathbb R)$ extension holds if and only if
$\widehat\mu\in L^s(\mathbb R)$, equivalently, if and only if
$\sum_nG_{s,n}<\infty$, where $G_{s,n}$ sums the $s$th powers of the
magnitudes of the $n$th digit mask at its nonzero grid frequencies. The proof combines finite-frame stability and multiscale
iteration with disjoint resonant frequency grids. The criterion
determines the critical restriction exponent and characterizes its
attainment when it lies in $(1,2)$.

We construct singular continuous spectral Moran measures with common
scales, digit cardinalities, exponential spectrum, maximal mask magnitudes
at nonzero grid frequencies, and a global Fourier decay bound sharp along
a common resonant sequence, yet different critical restriction exponents.
For any prescribed finite positive summability exponent of these maxima,
we determine the optimal interval of critical restriction
exponents and realize both endpoint attainment and failure at every
interior point with the same fixed data. These measures can simultaneously
satisfy $\dim_{\mathrm F}^{\theta}\mu=\theta$ for $0\le\theta\le1$ and
$D_q(\mu)=1$ for $1<q\le\infty$, with local dimension one at every
support point.

Finite-field spectral constructions realize every critical restriction
exponent in $[1,2]$ with these dimensional data and exhibit sharp
logarithmic endpoint transitions. They also yield singular continuous
measures with Fourier dimension zero, restriction for all $1\le p<2$,
and stretched-logarithmic Fourier decay with an optimal exponent.
Finally, all-scale $L^q$-dimension formulas yield explicit examples where
dimensional criteria strictly improve the classical
Stein--Tomas--Mockenhaupt range.
\end{abstract}
\maketitle

\setcounter{tocdepth}{1}
\tableofcontents

\section{Introduction and main results}

For a finite complex Borel measure $\nu$ on $\R$, we write
\[
\widehat\nu(\xi):=\int_{\R}e^{-2\pi i\xi x}\dd\nu(x),
\qquad \xi\in\R.
\]
For $f\in L^1(\R)$, we set $\widehat f:=\widehat{f\,dx}$.

For a Borel probability measure $\mu$ on $\R$, Fourier restriction asks for
the exponents $p\in[1,2]$ such that
$
\|\widehat f\|_{L^2(\mu)}\lesssim_p\|f\|_{L^p(\R)}$ for any $f\in\mathcal S(\R)$. For singular and fractal measures, restriction estimates reflect a delicate interaction between Fourier decay and geometric structure. Mockenhaupt and Mitsis showed that nontrivial restriction estimates follow from suitable Frostman bounds together with polynomial Fourier decay \cite{Mockenhaupt,Mitsis}, with the endpoint established by Bak--Seeger \cite{BakSeeger}. Sharpness examples of Hambrook--\L{}aba and Chen show that, in general,
no larger restriction range can be guaranteed from these two power
exponents alone \cite{HambrookLaba,ChenSharp}. This makes it natural to ask what finer structural information can govern sharp restriction. Indeed, additional convolution or arithmetic structure is known to improve restriction estimates \cite{ChenSeeger,LabaWang}.

Recent work has refined the dimensional information available to the
restriction problem. Carnovale, Fraser and de Orellana obtain stronger
sufficient estimates from the Fourier spectrum and from $L^q$-dimensions
\cite{CFdORestriction,CFdOLq}; Polasek and Rela \cite{PR} show that zero Fourier dimension
can still coexist with nontrivial restriction under suitable subpower decay and mass bounds, while Fraser's quantitative flatness gives complementary
obstructions \cite{FraserFlatness}. This leads to the following question:
\begin{quote}
	\emph{Is the optimal Fourier restriction range determined by the full
		Fourier spectrum together with the $L^q$-dimensions for all $q>1$, and if not, what
		finer multiscale quantity does?}
\end{quote}
This is the main motivation of the present paper.

We answer this question within the class of weighted Moran measures with
uniformly bounded one-level frame bounds. Higher moments of their one-level
Fourier coefficients give an exact restriction criterion. We then vary the
multiscale distribution of Fourier mass while keeping the scales, digit
cardinalities, maximal coefficient magnitudes, and other Fourier invariants
fixed.

Let $R_n\ge2$ be integers, put $Q_0=1$ and
$Q_n=R_1\cdots R_n$, and let $(w_{n,j})_{0\le j<R_n}$ be probability
vectors. We consider
\begin{equation}\label{eq:product-measure}
	\mu=\bigconv_{n=1}^{\infty}
	\left(\sum_{j=0}^{R_n-1}w_{n,j}\delta_{j/Q_n}\right).
\end{equation}
We refer to measures of the form \eqref{eq:product-measure} as
\emph{weighted Moran measures}. Unless otherwise specified, $\mu$ below
denotes such a measure. The associated Fourier masks are
\[
m_n(t)=\sum_{j=0}^{R_n-1}w_{n,j}e^{-2\pi ijt},
\qquad
D_{R_n}(t)=\frac1{R_n}\sum_{j=0}^{R_n-1}e^{-2\pi ijt},
\]
and we write
\begin{equation}\label{eq:one-level-quantities}
\begin{aligned}
\maskerr_n&:=\|m_n-D_{R_n}\|_\infty,
\qquad
C_n:=R_n\max_{0\le j<R_n}w_{n,j},
\qquad C_*:=\sup_{n\ge1}C_n,\\
\eta_n&:=\max_{1\le k<R_n}
\left|m_n\!\left(\frac{k}{R_n}\right)\right|,\qquad
G_{s,n}:=
\sum_{k=1}^{R_n-1}
\left|m_n\!\left(\frac{k}{R_n}\right)\right|^s,
\quad s\ge2.
\end{aligned}
\end{equation}
The quantity $G_{s,n}$ records the distribution of the nonzero Fourier mass
of the $n$th digit law, rather than only its largest coefficient.

We define the \emph{critical restriction exponent} of $\mu$ by
\begin{equation}\label{CRE}
	p_{\rm res}(\mu):=
	\sup\left\{p\in[1,2]:
	\|\widehat f\|_{L^2(\mu)}
	\lesssim_{p,\mu}\|f\|_{L^p(\R)}
	\text{ for all }f\in\mathcal S(\R)\right\}.
\end{equation}
Thus $p_{\rm res}(\mu)$ in \eqref{CRE} is the upper endpoint of the admissible restriction
range; the estimate need not hold at that endpoint. For $1<p\le2$, duality
identifies restriction at $p$ with
$L^2(\mu)\to L^{p'}(\R)$ extension, where $p'=p/(p-1)$.

We also define the \emph{critical Fourier-moment exponent} by
\begin{equation}\label{eq:critical-moment-exponent}
 s_c(\mu):=\inf\Bigl\{s>2:\sum_{n=1}^{\infty}G_{s,n}<\infty\Bigr\}
 \in[2,\infty],
\end{equation}
with the convention $\inf\varnothing=\infty$.

Our first result gives an exact criterion.

\begin{maintheoremA}
	[Theorem~\ref{thm:exact-moment-criterion} and
	Corollary~\ref{cor:critical-index}]
	Assume $C_*<\infty$. For every $2<s<\infty$, the following are
	equivalent:
\[
(i)~\|\widehat{g\mu}\|_{L^s(\R)}
	\lesssim_{s,\mu}\|g\|_{L^2(\mu)}
	~~(g\in L^2(\mu)),\quad
 (ii)~\widehat\mu\in L^s(\R), \quad
	(iii)~\sum_nG_{s,n}<\infty.
\]
With $s_c(\mu)$ as in \eqref{eq:critical-moment-exponent}, we have
	$
	p_{\rm res}(\mu)=\frac{s_c(\mu)}{s_c(\mu)-1},
	$
	with the convention that the right-hand side equals $1$ when $s_c(\mu)=\infty$.
	When $2<s_c(\mu)<\infty$, restriction holds at the critical exponent if
	and only if
	$
	\sum_nG_{s_c(\mu),n}<\infty.
	$
\end{maintheoremA}

Within this class, the single test $g\equiv1$ detects the entire
restriction range for $1<p<2$: restriction at $p$ holds if and only if
$\widehat\mu\in L^{p'}(\R)$. The endpoint $p=2$ is characterized
separately by $\prod_nC_n<\infty$; see \eqref{eq:operator-quadratic-endpoint}.
The moments $G_{s,n}$ record how Fourier mass is distributed among the
nonzero grid frequencies at one scale, while $\sum_nG_{s,n}$ records how
these contributions accumulate across scales. The constructions below
show that this summability is not determined by the dimension data or by
a common Fourier decay bound sharp in order.

The sufficient direction comes from a finite-frame stability estimate whose
error is controlled by an off-diagonal $s$th row moment, followed by a
scale-by-scale iteration. Necessity is obtained from disjoint resonant
frequency grids. One-level iteration has precedents in the almost-Parseval
frame towers of Lai--Wang \cite{LaiWang}; here it yields a
continuous-frequency $L^s$ extension estimate together with a matching
scalar necessity test. The construction of \L{}aba--Wang \cite{LabaWang}
combines finite $\Lambda(p)$ sets with decoupling to obtain near-optimal
restriction ranges. Our criterion determines the exact range for prescribed
digit laws, including attainment at critical exponents in $(1,2)$.

For explicit arithmetic realizations we use the finite-field digit sets
and finite Hadamard structure of Lai, Shi and Xie
\cite[Theorem~3.1 and Proposition~6.1]{LSX}. These digit sets underlie
Main Theorem B; the proofs of Main Theorems A and C do not require them.

For a tridiagonal specialization, we compute the one-level moments and
reduce the criterion to an explicit scalar series for arbitrary scale
sequences. The leading logarithmic growth of the digit parameters
determines the critical restriction exponent, while a lower-order
logarithmic correction determines its attainment. This yields the following
realizations with common dimension data. The dimension and spectral
notation used below is defined in Section~\ref{sec:dimension-notation}.

\begin{maintheoremB}
	[Theorem~\ref{thm:complete-nondeterminacy}]
	For every $p_0\in[1,2]$ there exists a singular continuous spectral Moran
	measure $\mu_{p_0}$ such that
	\[
	D_q(\mu_{p_0})=1\quad(1<q\le\infty),
	\qquad
	\dim_{\mathrm F}^{\theta}\mu_{p_0}=\theta
	\quad(0\le\theta\le1),
	\]
	and, moreover,
	\[
	\lim_{r\downarrow0}
	\frac{\log\mu_{p_0}(B(x,r))}{\log r}=1
	\qquad\text{for every }x\in\supp\mu_{p_0}.
	\]
	Nevertheless,
	$
	p_{\rm res}(\mu_{p_0})=p_0.
	$
	For each $1<p_0<2$, the measure can be chosen with either exact
	restriction range $[1,p_0]$ or $[1,p_0)$.
\end{maintheoremB}

The sufficient Fourier spectrum criterion of
\cite[Theorem~3.1]{CFdORestriction} requires
$\dim_{\mathrm F}^{\theta}\mu>\theta$. Main Theorem B lies on the boundary
$\dim_{\mathrm F}^{\theta}\mu=\theta$. For $0<\theta<1$, finiteness of
the critical Fourier integral is not recorded by this dimension value.
Under $C_*<\infty$, it is equivalent to
$\sum_nG_{2/\theta,n}<\infty$; see \eqref{eq:critical-line-moments}.

Main Theorem B allows the scales and digit cardinalities to depend on the
restriction exponent. Main Theorem C fixes these data, an exponential
spectrum, the dimensional data, and a common sharp Fourier bound, while
the restriction exponent still varies. Recall the maxima $\eta_n$ from
\eqref{eq:one-level-quantities}.

\begin{maintheoremC}
	[Theorem~\ref{thm:fixed-envelope}]
	There are fixed sequences $R_n=2M_n$, a fixed sequence
	$\eta_n\asymp n^{-1/2}$, and a common exponential spectrum $\Lambda$ such that, for every
	$p_0\in[4/3,2]$, one can choose equal-weighted digit sets of cardinality
	$M_n$ producing a singular continuous spectral Moran measure $\mu_{p_0}$
	with
	\[
	p_{\rm res}(\mu_{p_0})=p_0,
	\qquad
	D_q(\mu_{p_0})=1\quad(1<q\le\infty),
	\qquad
	\dim_{\mathrm F}^{\theta}\mu_{p_0}=\theta
	\quad(0\le\theta\le1).
	\]
	All members have pointwise local dimension one everywhere on their supports,
	and their maximal mask magnitudes at nonzero grid frequencies agree at every
	level. Moreover, with a constant independent of $p_0$,
	\[
	|\widehat\mu_{p_0}(\xi)|
	\lesssim
	\frac{\log\log(e^e+|\xi|)}
	{\sqrt{\log(e+|\xi|)}},
	\]
and this common Fourier decay bound is sharp along the same resonant frequencies $\xi_n=Q_n/2$.
\end{maintheoremC}

These common data are chosen as part of the construction. They do not
determine the restriction range: the remaining freedom lies in the
distribution of Fourier mass below the maximal coefficient at each level.

There is also a sharp obstruction when $(\eta_n)$ is fixed. Put
$\tau:=\inf\{t>0:\sum_n\eta_n^t<\infty\}\in[0,\infty]$.
Under $C_*<\infty$, Proposition~\ref{prop:fixed-envelope-obstruction} gives
$\max\{2,\tau\}\le s_c(\mu)\le\tau+2$ when $\tau<\infty$;
if $\tau=\infty$, then $s_c(\mu)=\infty$ and $p_{\rm res}(\mu)=1$.
Main Theorem C realizes the full interval allowed by this obstruction when
$\eta_n\asymp n^{-1/2}$. Theorem~\ref{thm:fixed-envelope-general-tau}
determines the optimal interval for every $0<\tau<\infty$, and
Corollary~\ref{cor:fixed-endpoint-separation} realizes both attainment and
failure of the critical endpoint throughout its interior with the same
fixed data.

The later sections compare this exact multiscale criterion with more
traditional Fourier and mass-dimensional information. We obtain formulas
for Fourier and correlation dimensions of weighted Moran measures, as well
as all-scale $L^q$-dimension formulas that account for intermediate scales
between consecutive cylinders. These results also yield examples where
$L^q$-dimensional information improves the classical
Stein--Tomas--Mockenhaupt range.

Throughout, $A\lesssim B$ means $A\le CB$ for a constant $C>0$, and
$A\asymp B$ means that both $A\lesssim B$ and $B\lesssim A$ hold. A
subscript indicates dependence of the implicit constant on the displayed
parameter.

The paper is organized as follows.
Section~\ref{sec:criterion} proves the exact moment criterion.
Section~\ref{sec:fixed-envelope} establishes the fixed-data
non-determinacy results, including the optimal range and endpoint
separation.
Section~\ref{sec:restriction} gives finite-field spectral realizations and
the logarithmic threshold results.
Finally, Sections~\ref{sec:general-proof} and~\ref{sec:weil} develop the
auxiliary Fourier, correlation and $L^q$-dimension formulas and compare
their restriction consequences with the exact moment criterion.

\section{An exact Fourier-moment criterion}\label{sec:criterion}

We prove Main Theorem A using the one-level Fourier moments. The only
structural hypothesis for the equivalence is $C_*<\infty$.
The sufficiency argument combines one-level bounds, while the necessity argument evaluates the Fourier transform on carefully chosen disjoint sets of resonant frequencies.

\subsection{Product representation and normalized tails}

Let $(J_n)_{n\ge1}$ be independent random variables with
$
\Prob(J_n=j)=w_{n,j},\quad 0\le j<R_n.
$
Then $\mu$ is the law of
$
X=\sum_{n=1}^\infty \frac{J_n}{Q_n}.
$
The infinite convolution in \eqref{eq:product-measure} is well defined since
\[
\sum_{n=1}^{\infty}\frac{R_n-1}{Q_n}
=\sum_{n=1}^{\infty}\left(\frac1{Q_{n-1}}-\frac1{Q_n}\right)=1,
\]
so the associated random series converges absolutely and is supported in
$[0,1]$.

We write $\unif$ for Lebesgue probability measure on $[0,1]$.
For $n\ge0$, the part of the infinite convolution defining $\mu$ after level $n$, rescaled to the unit interval, is
\[
Z_n=\sum_{j>n}\frac{J_j}{R_{n+1}\cdots R_j}\in[0,1].
\]
Its distribution is
\[
\nu_n
=
\bigconv_{j=n+1}^{\infty}
\left(
\sum_{k=0}^{R_j-1}
w_{j,k}\,
\delta_{\,k/(R_{n+1}\cdots R_j)}
\right).
\]

\begin{lemma}\label{lem:bounded-frame-tail}\label{lem:tail}
Suppose $C_*<\infty$, with $C_*$ as in
\eqref{eq:one-level-quantities}. Then
\begin{equation}\label{eq:bounded-frame-tail}
 \inf_{n\ge0}\inf_{|u|\le1/2}|\widehat\nu_n(u)|
 \ge c(C_*):=\frac1{2\sqrt2}\sin\!\left(\frac{\pi}{4C_*}\right)>0.
\end{equation}

Alternatively, if $\varepsilon_n\to0$, then for all sufficiently large $n$,
$
 \inf_{|u|\le1/2}|\widehat\nu_n(u)|\ge\frac1{\sqrt2\,\pi}.
$
\end{lemma}
\begin{proof}
Fix $n\ge0$, put $R=R_{n+1}$, and let $|u|\le1/2$. Splitting
the first digit from the normalized tail gives
$
 \widehat\nu_n(u)
 =m_{n+1}(u/R)\,\widehat\nu_{n+1}(u/R).
$
We first bound the second factor. If $\sigma$ is any probability measure
supported on $[0,1]$, then 
$
	|\widehat\sigma(t)|
	\ge
	\operatorname{Re}\!\left(e^{\pi it}\widehat\sigma(t)\right)
	\ge \cos(\pi|t|)$ for 
$ |t|\le\frac12.$ 
Since $R\ge2$, we have $|u/R|\le1/4$. Applying this estimate to
$\sigma=\nu_{n+1}$ gives
\[
 |\widehat\nu_{n+1}(u/R)|\ge\cos(\pi|u|/R)\ge \cos\!\left(\frac{\pi}{4}\right)
 =\frac1{\sqrt2}.
\]

To prove \eqref{eq:bounded-frame-tail}, let $U$ be independent and uniform
on $[0,1]$, and let $\rho$ be the law of $(J_{n+1}+U)/R$. On each
interval $[j/R,(j+1)/R)$, the density of $\rho$ is $Rw_{n+1,j}$ and
is therefore at most $C_*$. Since the weights sum to one, $C_*\ge1$.
Set $a=(4C_*)^{-1}$; the two end intervals have total length $2a$,
so
$
 \rho([a,1-a])\ge1-2aC_*=\frac12.
$
For $x\in[0,1]$, the centered cosine
$\cos(2\pi u(x-1/2))$ is nonnegative. Moreover, if
$x\in[a,1-a]$, then
$
 |2\pi u(x-1/2)|\le\pi(1/2-a)$ and $
 \cos\bigl(2\pi u(x-1/2)\bigr)\ge\sin(\pi a).
$
Consequently,
\[
 \begin{aligned}
 |\widehat\rho(u)|
 &\ge \operatorname{Re}\bigl(e^{\pi iu}\widehat\rho(u)\bigr)\\
 &=\int\cos\bigl(2\pi u(x-1/2)\bigr)\,\dd\rho(x)
 \ge \rho([a,1-a])\sin(\pi a)
 \ge\frac12\sin(\pi a).
 \end{aligned}
\]
Since $\widehat\rho(u)=m_{n+1}(u/R)\widehat\unif(u/R)$ and
$|\widehat\unif|\le1$, we have
\[
 |m_{n+1}(u/R)|\ge|\widehat\rho(u)|\ge\frac12\sin(\pi a).
\]
Combining the two factors gives, uniformly in $n$ and $|u|\le1/2$,
\[
 |\widehat\nu_n(u)|
 \ge\frac1{2\sqrt2}\sin\!\left(\frac{\pi}{4C_*}\right)
 =c(C_*).
\]

For the alternative assertion, assume $\varepsilon_n\to0$.
For $0<|u|\le1/2$,
the geometric-sum formula and $|\sin v|\le|v|$ give
\[
 \left|D_R(u/R)\right|
 =\frac{|\sin(\pi u)|}{R|\sin(\pi u/R)|}
 \ge\frac{|\sin(\pi u)|}{\pi|u|}
 \ge\frac2\pi.
\]
The same bound holds at $u=0$ by continuity. Using
$\varepsilon_{n+1}=\|m_{n+1}-D_R\|_\infty$, we obtain
\[
 |m_{n+1}(u/R)|
 \ge\left|D_R(u/R)\right|-\varepsilon_{n+1}
 \ge\frac2\pi-\varepsilon_{n+1}.
\]
For all sufficiently large $n$, $\varepsilon_{n+1}\le1/\pi$.
The tail factorization and the bound above therefore yield
\[
 |\widehat\nu_n(u)|
 \ge\left(\frac2\pi-\varepsilon_{n+1}\right)\frac1{\sqrt2}
 \ge\frac1{\sqrt2\,\pi}.\qedhere
\]
\end{proof}

\subsection{Iteration of one-level estimates}

For a weighted Moran measure as in \eqref{eq:product-measure}, recall
$\maskerr_n$, $C_n$, and $\eta_n$ from \eqref{eq:one-level-quantities}.
Since $D_{R_n}(k/R_n)=0$ for $1\le k<R_n$, one always has $\eta_n\le\maskerr_n$.

The following iteration principle derives a continuous-frequency
extension estimate from suitable one-level estimates.
\begin{proposition}\label{thm:frame-criterion}
Let $s>2$. Suppose that for each $n\ge1$ there is a constant
$K_n\ge1$ such that, for every vector
$h=(h_b)_{b=0}^{R_n-1}\in\mathbb C^{R_n}$ and every $u\in\R$,
\begin{equation}\label{eq:level-iteration-assumption}
 \sum_{k=0}^{R_n-1}\left|
 \sum_b w_{n,b}h_be^{-2\pi i b(k+u)/R_n}\right|^s
 \le K_n\left(\sum_b w_{n,b}|h_b|^2\right)^{s/2}.
\end{equation}
If $\prod_nK_n<\infty$, then
\begin{equation}\label{eq:frame-extension}
 \|\widehat{g\mu}\|_{L^s(\R)}
 \le\left(\prod_nK_n\right)^{1/s}\|g\|_{L^2(\mu)}
 \qquad(g\in L^2(\mu)).
\end{equation}
Consequently restriction holds at $p=s/(s-1)$.
\end{proposition}

\begin{proof}
Let $\mu_{\le N}
:=
\bigconv_{n=1}^{N}
\left(\sum_{j=0}^{R_n-1}w_{n,j}\delta_{j/Q_n}\right).$
Distinct finite digit prefixes yield distinct points
$\sum_{\ell=1}^N j_\ell/Q_\ell$ in $Q_N^{-1}\mathbb Z$;
we identify functions of these prefixes with functions on their images.
For a function $g$ of these digits and $0\le b<R_N$, define
\[
F_b(\xi):=
\sum_{\substack{0\le j_\ell<R_\ell\\1\le\ell\le N-1}}
\left(\prod_{\ell=1}^{N-1}w_{\ell,j_\ell}\right)
g(j_1,\ldots,j_{N-1},b)
\exp\!\left(-2\pi i\xi\sum_{\ell=1}^{N-1}\frac{j_\ell}{Q_\ell}\right).
\]
For $N=1$, the sum has one term indexed by the empty tuple, so $F_b(\xi)=g(b)$.
For $0\le\xi<Q_{N-1}$ and $0\le k<R_N$, periodicity of the first
$N-1$ levels gives
\[
\widehat{g\mu_{\le N}}(\xi+kQ_{N-1})
=\sum_{b=0}^{R_N-1}w_{N,b}F_b(\xi)
e^{-2\pi i b(k+\xi/Q_{N-1})/R_N}.
\]
Apply \eqref{eq:level-iteration-assumption} with $h_b=F_b(\xi)$, sum over $k$, and integrate in $\xi$. Minkowski's inequality in $L^{s/2}$ then gives
\[
\int_0^{Q_N}|\widehat{g\mu_{\le N}}(\xi)|^s\,\dd\xi
\le K_N\left(\sum_b w_{N,b}
\left(\int_0^{Q_{N-1}}|F_b(\xi)|^s\,\dd\xi\right)^{2/s}\right)^{s/2}.
\]
Applying the same estimate inductively to each $F_b$ proves
\begin{equation}\label{eq:finite-frame-induction}
\int_0^{Q_N}|\widehat{g\mu_{\le N}}(\xi)|^s\,\dd\xi
\le \left(\prod_{n\le N}K_n\right)\|g\|_{L^2(\mu_{\le N})}^s.
\end{equation}

We next control the normalized tail without truncating it to an atomic measure. For $M\ge N$, let $\nu_N^{(M)}$ be the law of
\[
\sum_{j=N+1}^{M}\frac{J_j}{R_{N+1}\cdots R_j}
+\frac{U}{R_{N+1}\cdots R_M},
\]
where $U$ is independent and uniform on $[0,1]$; for $M=N$ this is Lebesgue measure on $[0,1]$. Parseval for the Fourier series of $e^{-2\pi iux}$ on $[0,1]$ gives
$\sum_{j\in\Z}|\widehat\unif(u+j)|^2=1$, hence the same sum with exponent $s\ge2$ is at most one. Put
$S_{N,M}=\sup_u\sum_{j\in\Z}|\widehat{\nu_N^{(M)}}(u+j)|^s$.
For $N<M$, split $j=R_{N+1}l+k$, $0\le k<R_{N+1}$. The factorization of $\widehat{\nu_N^{(M)}}$ and the periodicity of $m_{N+1}$ show that
\[
\sum_{j\in\Z}|\widehat{\nu_N^{(M)}}(u+j)|^s
\le S_{N+1,M}\sum_{k=0}^{R_{N+1}-1}
\left|m_{N+1}\!\left(\frac{u+k}{R_{N+1}}\right)\right|^s
\le K_{N+1}S_{N+1,M},
\]
where the last inequality is \eqref{eq:level-iteration-assumption} with $h\equiv1$. Since $S_{M,M}\le1$, iteration gives $S_{N,M}\le\prod_{n=N+1}^M K_n$. Now $\nu_N^{(M)}\Rightarrow\nu_N$; pointwise convergence of Fourier transforms and Fatou's lemma yield
\begin{equation}\label{eq:tail-sampling-frame}
\sup_{u\in\R}\sum_{j\in\Z}|\widehat{\nu_N}(u+j)|^s
\le\prod_{n>N}K_n.
\end{equation}

To pass to the full measure without imposing uniqueness of digit
expansions, work on the product space
\[
\Omega=\prod_{n\ge1}\{0,\ldots,R_n-1\},
\qquad
\mathbb P=\bigotimes_{n\ge1}(w_{n,j})_{j=0}^{R_n-1},
\]
and let $\pi(\omega)=\sum_{n\ge1}\omega_n/Q_n$, so that
$\mu=\pi_\#\mathbb P$. Let $H:\Omega\to\mathbb C$ depend only on
the first $N$ coordinates, and put
\[
 \pi_N(\omega)=\sum_{n=1}^N\frac{\omega_n}{Q_n},
 \qquad \sigma_N=(\pi_N)_\#(H\mathbb P).
\]
Here $H\mathbb P$ denotes the complex measure with density $H$
with respect to $\mathbb P$.
The estimate \eqref{eq:finite-frame-induction} applies to $\sigma_N$,
with the right-hand norm equal to $\|H\|_{L^2(\mathbb P)}$.
Independence of the remaining digits gives
\[
 \widehat{\pi_\#(H\mathbb P)}(\xi)
 =\widehat{\sigma_N}(\xi)\widehat{\nu_N}(\xi/Q_N).
\]
Decomposing $\R$ into translates of $[0,Q_N)$ and using the
$Q_N$-periodicity of $\widehat{\sigma_N}$, we obtain
\[
 \begin{aligned}
 \bigl\|\widehat{\pi_\#(H\mathbb P)}\bigr\|_s^s
 &=\int_0^{Q_N}|\widehat{\sigma_N}(\xi)|^s
   \sum_{j\in\Z}\left|\widehat{\nu_N}\!\left(\frac{\xi}{Q_N}+j\right)\right|^s
   \dd\xi\\
 &\le\left(\prod_{n>N}K_n\right)
   \int_0^{Q_N}|\widehat{\sigma_N}(\xi)|^s\dd\xi\\
 &\le\left(\prod_{n\ge1}K_n\right)\|H\|_{L^2(\mathbb P)}^s,
 \end{aligned}
\]
where the two inequalities follow from \eqref{eq:tail-sampling-frame}
and \eqref{eq:finite-frame-induction}, respectively.

Now let $g\in L^2(\mu)$, put $G=g\circ\pi$, and set
$G_N=\mathbb E(G\mid\mathcal F_N)$, where $\mathcal F_N$ is generated
by the first $N$ coordinates. Then $G_N$ depends only on these
coordinates and $G_N\to G$ in $L^2(\mathbb P)$.
The estimate just proved applies to
the pushed-forward complex measure $\pi_\#(G_N\mathbb P)$ and gives
\[
\bigl\|\widehat{\pi_\#(G_N\mathbb P)}\bigr\|_s
\le
\left(\prod_{n\ge1}K_n\right)^{1/s}
\|G_N\|_{L^2(\mathbb P)}.
\]
For each $\xi$, the corresponding Fourier transforms converge to
$\widehat{g\mu}(\xi)$ because $G_N\to G$ in $L^1(\mathbb P)$. Fatou's
lemma and $\|G_N\|_2\to\|G\|_2=\|g\|_{L^2(\mu)}$ prove
\eqref{eq:frame-extension}. Finally, duality between restriction at
$p$ and extension at $p'=s$ gives the last assertion.
\end{proof}

\subsection{An exact one-level moment criterion}\label{sec:exact-moments}

We now estimate the one-level constants using the full Fourier moments
$G_{s,n}$, $s>2$, defined in \eqref{eq:one-level-quantities}.
Under bounded one-level frame bounds, summability of these moments is
both necessary and sufficient for restriction. The following stability
estimate controls the excess above one by the off-diagonal row moments.

\begin{lemma}
\label{lem:q41-frame}
Let $s>2$ and $C_0\ge1$. There exists
$\eta_0=\eta_0(s,C_0)>0$ with the following property.
Let $\mathcal H$ be a complex Hilbert space and let
$\phi_1,\ldots,\phi_R\in\mathcal H$ be unit vectors. Assume that, for every
$h\in\mathcal H$,
\[
 \sum_{k=1}^R|\langle h,\phi_k\rangle|^2\le C_0\|h\|^2,
 \qquad
 |\langle\phi_k,\phi_l\rangle|\le\eta_0\quad(k\ne l).
\]
Set
$
 \Gamma_s=\max_j\sum_{k\ne j}|\langle\phi_j,\phi_k\rangle|^s.
$
Then, for every $h\in\mathcal H$,
\begin{equation}\label{eq:q41-frame-moment}
 \sum_{k=1}^R|\langle h,\phi_k\rangle|^s
 \le(1+2^{s-1}\Gamma_s)\|h\|^s.
\end{equation}
\end{lemma}

\begin{proof}
Put $a=s/2>1$. Choose $0<\delta_0<1$ such that
\begin{equation}\label{eq:q41-delta-choice}
 2^{s-1}C_0\delta_0^{a-1}\le1.
\end{equation}
Next choose sufficiently small $0<\tau<1$ and then sufficiently small $\eta_0>0$ so that
\[
 (1-\delta_0)^{a-1}\bigl(1+(C_0/\tau)\eta_0\bigr)
       +C_0\tau^{a-1}<1.
\]
This is possible since $(1-\delta_0)^{a-1}<1$; all choices depend
only on $s$ and $C_0$.

The case $h=0$ is immediate; by homogeneity, assume $\|h\|=1$.
Set $y_k=|\langle h,\phi_k\rangle|^2$ and $t=\max_k y_k$.
The unit norms and the frame bound give
$0\le y_k\le1$ and $\sum_k y_k\le C_0$.

We first claim that $\sum_k y_k^a>1$ forces $t>1-\delta_0$.
Suppose $t\le1-\delta_0$, and put $S=\{k:y_k>\tau\}$.
Then $|S|\le C_0/\tau$. For $S\ne\varnothing$, let $G_S$ be the
Gram matrix of $(\phi_k)_{k\in S}$. The analysis map
$h\mapsto(\langle h,\phi_k\rangle)_{k\in S}$ has squared operator norm
$\|G_S\|_{\ell^2\to\ell^2}$, so the row-sum bound gives
\[
 \sum_{k\in S}y_k
 \le\|G_S\|_{\ell^2\to\ell^2}
 \le1+(|S|-1)\eta_0
 \le1+(C_0/\tau)\eta_0.
\]
The same upper bound for $\sum_{k\in S}y_k$ is immediate if
$S=\varnothing$. Splitting the sum over $S$ and its complement gives
\begin{align*}
 \sum_k y_k^a
 &\le (1-\delta_0)^{a-1}\sum_{k\in S}y_k
       +\tau^{a-1}\sum_{k\notin S}y_k\\
 &\le (1-\delta_0)^{a-1}\bigl(1+(C_0/\tau)\eta_0\bigr)
       +C_0\tau^{a-1}<1,
\end{align*}
which proves the claim.

If $\sum_k y_k^a\le1$, the desired estimate already holds.
Otherwise, choose $j$ with $y_j=t$ and put
$\delta:=1-t\in[0,\delta_0)$. Here $\delta_0$ is fixed, whereas
$\delta$ depends on $h$. Decompose $h$ orthogonally as
\[
 h=z\phi_j+v,\qquad v\perp\phi_j,\qquad
 |z|^2=1-\delta,\quad\|v\|^2=\delta.
\]
By Cauchy--Schwarz and the upper frame bound,
\[
 \sum_k|\langle v,\phi_k\rangle|^s
 \le\|v\|^{s-2}\sum_k|\langle v,\phi_k\rangle|^2
 \le C_0\|v\|^s=C_0\delta^a.
\]
For $k\ne j$, apply $|u+w|^s\le2^{s-1}(|u|^s+|w|^s)$ to
$\langle h,\phi_k\rangle$. Separating the $j$th term and using
$|z|\le1$ and the definition of $\Gamma_s$, we obtain
\begin{align*}
 \sum_k|\langle h,\phi_k\rangle|^s
 &\le (1-\delta)^a
       +2^{s-1}|z|^s\sum_{k\ne j}|\langle\phi_j,\phi_k\rangle|^s
       +2^{s-1}\sum_{k\ne j}|\langle v,\phi_k\rangle|^s\\
 &\le (1-\delta)^a+2^{s-1}\Gamma_s+2^{s-1}C_0\delta^a\\
 &\le 1-\delta+2^{s-1}\Gamma_s+\delta=1+2^{s-1}\Gamma_s.
\end{align*}
The last step uses $(1-\delta)^a\le1-\delta$ and
$2^{s-1}C_0\delta^a\le\delta$, the latter following from
\eqref{eq:q41-delta-choice} and $\delta<\delta_0$.
\end{proof}

\begin{corollary}\label{cor:q41-one-level}
Suppose $C_n\le C_0<\infty$ for all $n$ and $\eta_n\to0$.
For every fixed $s>2$ and all sufficiently large $n$, the following
estimate holds for every $u\in\R$ and every vector
$h=(h_b)_{b=0}^{R_n-1}\in\mathbb C^{R_n}$:
\begin{equation}\label{eq:q41-one-level}
 \sum_{k=0}^{R_n-1}\left|
 \sum_b w_{n,b}h_b{e}^{-2\pi i b(k+u)/R_n}\right|^s
 \le(1+2^{s-1}G_{s,n})
       \left(\sum_b w_{n,b}|h_b|^2\right)^{s/2}.
\end{equation}
The optimal constant in \eqref{eq:q41-one-level} lies between
$1+G_{s,n}$ and $1+2^{s-1}G_{s,n}$.
\end{corollary}

\begin{proof}
Let $w_n=\sum_{b=0}^{R_n-1}w_{n,b}\delta_b$ be the probability
measure on the digit set. We use $L^2(w_n)$ with inner product
\[
 \langle h,f\rangle_{L^2(w_n)}
 =\sum_{b=0}^{R_n-1}w_{n,b}h_b\overline{f_b}.
\]
Fix $u\in\R$ and set $\phi_k(b)=e^{2\pi i b(k+u)/R_n}$,
$0\le k<R_n$. Then $\|\phi_k\|_{L^2(w_n)}=1$, and discrete Parseval gives
\[
 \begin{aligned}
 \sum_{k=0}^{R_n-1}|\langle h,\phi_k\rangle|^2
 =R_n\sum_b w_{n,b}^2|h_b|^2\le C_n\|h\|_{L^2(w_n)}^2
 \le C_0\|h\|_{L^2(w_n)}^2.
 \end{aligned}
\]
Moreover, $\langle\phi_k,\phi_l\rangle=m_n((l-k)/R_n)$,
and every off-diagonal row moment is exactly $G_{s,n}$.
Since $\eta_n\le\eta_0(s,C_0)$ for all sufficiently large $n$,
Lemma~\ref{lem:q41-frame}, specifically \eqref{eq:q41-frame-moment}, gives the upper bound uniformly in $u$.

For the lower bound, take the unit vector $h=\phi_j$. Then
\[
 \sum_{k=0}^{R_n-1}|\langle\phi_j,\phi_k\rangle|^s
 =1+\sum_{k\ne j}|\langle\phi_j,\phi_k\rangle|^s
 =1+G_{s,n}.\qedhere
\]
\end{proof}

The two-sided comparison in Corollary~\ref{cor:q41-one-level} identifies
$G_{s,n}$ as the relevant one-level quantity, and the next theorem shows
that its summability precisely characterizes the global extension estimate.
\begin{theorem}\label{thm:exact-moment-criterion}
Assume $C_*<\infty$, with $C_*$ as in
\eqref{eq:one-level-quantities}.
For every $2<s<\infty$, the following assertions are equivalent:
\begin{align}
 &\|\widehat{g\mu}\|_{L^s(\R)}\lesssim_{s,\mu}\|g\|_{L^2(\mu)}
        \quad\text{for all }g\in L^2(\mu);\label{eq:q41-crit-a}\\
 &\widehat\mu\in L^s(\R);\label{eq:q41-crit-b}\\
 &\sum_{n=1}^{\infty}G_{s,n}<\infty.\label{eq:q41-crit-c}
\end{align}
The implicit constant in \eqref{eq:q41-crit-a} may depend on $s$ and
$\mu$, but not on $g$. Thus, within this class, the test $g\equiv1$
detects the entire restriction range for $1<p<2$.
\end{theorem}

\begin{proof}
The implication \eqref{eq:q41-crit-a}$\Rightarrow$\eqref{eq:q41-crit-b} is
immediate. To prove \eqref{eq:q41-crit-c} $\Rightarrow$
\eqref{eq:q41-crit-a}, note that
$\eta_n^s\le G_{s,n}\to0$, so $\eta_n\to0$.
By Corollary~\ref{cor:q41-one-level}, applied with $C_0=C_*$,
the one-level constants can be chosen as
$
 K_n=1+2^{s-1}G_{s,n}
$
for all sufficiently large $n$. 
For the finitely many remaining levels, Cauchy--Schwarz and discrete
Parseval give \eqref{eq:level-iteration-assumption} with $K_n=C_n$:
each Fourier coefficient has modulus at most $\|h\|_{L^2(w_n)}$, and
the sum of its squared moduli is at most $C_n\|h\|_{L^2(w_n)}^2$.
Thus $\prod_nK_n<\infty$, and Proposition~\ref{thm:frame-criterion}
proves \eqref{eq:q41-crit-a}.

It remains to prove \eqref{eq:q41-crit-b} $\Rightarrow$
\eqref{eq:q41-crit-c}. Since $\supp\mu\subset[0,1]$, choose a smooth
compactly supported $\chi$ equal to one on $[0,1]$, and put
$\psi=\widehat\chi$. Then $\widehat\mu=\widehat\mu\conv\psi$.
H\"older's inequality and the rapid decay of $\psi$ give
\begin{equation}\label{eq:q41-sampling}
 \sum_{j\in\Z}|\widehat\mu(j)|^s
 \le\|\psi\|_1^{s-1}
 \sup_t\sum_{j\in\Z}|\psi(j-t)|\,
 \|\widehat\mu\|_s^s<\infty.
\end{equation}
Lemma~\ref{lem:bounded-frame-tail} gives a constant $c=c(C_*)>0$
such that $|\widehat\nu_n(t)|\ge c$ for $|t|\le1/2$ and every $n$.
For integers $k$ satisfying $1\le|k|\le R_n/2$, all preceding
masks at $kQ_{n-1}$ are one. Therefore
$
 |\widehat\mu(kQ_{n-1})|
 =|m_n(k/R_n)|\,|\widehat\nu_n(k/R_n)|
 \ge c|m_n(k/R_n)|.
$
By conjugate symmetry,
$
 \left|m_n\!\left(\frac{R_n-k}{R_n}\right)\right|
 =\left|m_n\!\left(\frac{k}{R_n}\right)\right|,
$
so
\[
 G_{s,n}
 \le 2\sum_{1\le k\le\lfloor R_n/2\rfloor}
 \left|m_n\!\left(\frac{k}{R_n}\right)\right|^s,
\]
where the residue $R_n/2$ is only harmlessly overcounted when $R_n$ is even.
The preceding lower bound therefore gives
\begin{equation}\label{eq:q41-resonant-grid-bound}
 G_{s,n}
 \lesssim_{s,C_*}
 \sum_{1\le k\le\lfloor R_n/2\rfloor}
 |\widehat\mu(kQ_{n-1})|^s.
\end{equation}
The frequency sets on the right are pairwise disjoint as $n$ varies, since
they lie in $[Q_{n-1},Q_n/2]$. Summing
\eqref{eq:q41-resonant-grid-bound} in $n$ and using
\eqref{eq:q41-sampling} proves
$
 \sum_{n\ge1}G_{s,n}
 \lesssim_{s,C_*}
 \sum_{j\in\Z}|\widehat\mu(j)|^s<\infty.
$
\end{proof}

\begin{corollary}\label{cor:critical-index}
Under the assumptions of Theorem~\ref{thm:exact-moment-criterion}, the
critical Fourier-moment exponent $s_c(\mu)$ defined in
\eqref{eq:critical-moment-exponent} satisfies
\begin{equation}\label{eq:critical-index-formula}
 p_{\rm res}(\mu)=\frac{s_c(\mu)}{s_c(\mu)-1},
\end{equation}
with the convention that the right-hand side equals $1$ when $s_c(\mu)=\infty$.  If $2<s_c(\mu)<\infty$, restriction holds at
$p=s_c/(s_c-1)$ if and only if
$\sum_nG_{s_c,n}<\infty$.
\end{corollary}

\begin{proof}
Since $|m_n(k/R_n)|\le1$, the quantities $G_{s,n}$ are nonincreasing
in $s$. Hence the set of $s>2$ for which $\sum_nG_{s,n}<\infty$ is
either empty or an upper interval, possibly open at its left endpoint. By
Theorem~\ref{thm:exact-moment-criterion}, extension holds at exponent
$s>2$ exactly on this set.

Restriction at $p=1$ always holds because $\mu$ is a probability
measure. If restriction holds at $p=2$, interpolation with the
$p=1$ estimate gives restriction for every $1<p<2$, and hence
$s_c=2$ by the theorem. Thus $p=2$ cannot be an isolated admissible
exponent. If $s_c=\infty$, restriction therefore holds only at
$p=1$. If $s_c=2$, the upper-interval property gives extension for
every $s>2$, so restriction holds for every $1<p<2$ and
$p_{\rm res}=2$; this does not assert attainment at $p=2$.

For $2<s_c<\infty$, duality, with $p=s'=s/(s-1)$, gives
\eqref{eq:critical-index-formula}. The same theorem shows that the
critical endpoint is attained exactly when the series at $s=s_c$
converges. Together with the two boundary cases above, this proves
all the assertions.
\end{proof}

\subsection{The uniform one-level bound and the endpoint \(s=2\)}
 The uniform frame bound \(C_*<\infty\) in
Theorem~\ref{thm:exact-moment-criterion} cannot simply be omitted:
even if both \(\widehat\mu\in L^s(\mathbb R)\) and
\(\sum_nG_{s,n}<\infty\) hold, the extension estimate may fail.
\begin{example}\label{ex:unbounded-frame-counterexample}
Let $R_n=2^n$, $\alpha_n=R_n^{-3/8}$, and set
\[
 w_{n,0}=\alpha_n+\frac{1-\alpha_n}{R_n},\qquad
 w_{n,j}=\frac{1-\alpha_n}{R_n}\quad(1\le j<R_n).
\]
The resulting weighted Moran measure satisfies
$\widehat\mu\in L^4(\R)$ and $\sum_nG_{4,n}<\infty$, but it has no
$L^2(\mu)\to L^4(\R)$ extension estimate.
\end{example}
\begin{proof}
At every nonzero grid frequency $m_n(k/R_n)=\alpha_n$, so
\[
 G_{4,n}=(R_n-1)\alpha_n^4\asymp R_n^{-1/2},\qquad
 C_n=1-\alpha_n+R_n\alpha_n\asymp R_n^{5/8}\to\infty.
\]
In particular, $\sum_nG_{4,n}<\infty$ since $R_n=2^n$.
We use the following uniform-tail approximation, which is defined for
any weighted Moran measure. Let $(J_n)$ be its independent digit
variables, and let $U$ be independent of them and uniformly distributed
on $[0,1]$. For $N\ge1$, let $f_N$ be the density of
\[
 X_N=\sum_{n=1}^N\frac{J_n}{Q_n}+\frac{U}{Q_N}.
\]
Thus $f_N\,dx=\mu_{\le N}*(Q_N\mathbf 1_{[0,Q_N^{-1}]}\,dx)$, where
$\mu_{\le N}$ is the discrete prefix measure from the proof of
Proposition~\ref{thm:frame-criterion}. The remaining tail is replaced
by the uniform law on $[0,Q_N^{-1}]$.
For a digit sequence $\mathbf j=(j_1,\ldots,j_N)$ with
$0\le j_n<R_n$, set
\[
 I_{\mathbf j}=\left[
 \sum_{n=1}^N\frac{j_n}{Q_n},\,
 \sum_{n=1}^N\frac{j_n}{Q_n}+\frac1{Q_N}
 \right).
\]
These intervals partition $[0,1)$, and the density is given by
\[
 f_N(x)=Q_N\prod_{n=1}^Nw_{n,j_n}
 \quad(x\in I_{\mathbf j}),
\]
with $f_N=0$ outside $[0,1)$. With respect to Lebesgue measure on
$[0,1)$, these densities form a nonnegative martingale on the nested
cylinder partitions. Moreover, $f_N\,dx\Rightarrow\mu$, since
$|X_N-X|\le Q_N^{-1}$, where $X=\sum_{n\ge1}J_n/Q_n$.
For the weights in this example and $q=4/3$, summing over the cylinders
gives
\[
 \|f_N\|_q^q=\prod_{n\le N}\left(R_n^{q-1}\sum_jw_{n,j}^q\right),
 \qquad
 R_n^{q-1}\sum_jw_{n,j}^q\le1+C R_n^{-1/6}.
\]
Indeed, the digits $j\ne0$ contribute at most $1-R_n^{-1}$, while
$(\alpha_n+R_n^{-1})^q\le2^{q-1}(\alpha_n^q+R_n^{-q})$ and
$R_n^{q-1}\alpha_n^q=R_n^{-1/6}$ control the remaining term. Since
$\sum_nR_n^{-1/6}<\infty$, the $L^{4/3}$ norms are uniformly bounded.

The martingale convergence theorem gives $f_N\to f$ in $L^{4/3}$;
the weak convergence above then identifies $\mu=f\,dx$.
The Hausdorff--Young inequality gives $\widehat\mu\in L^4$.

Since $\mu=f\,dx$, the cylinder endpoints have zero $\mu$-mass, and
\[
 \mu([0,Q_N^{-1}))=\prod_{n\le N}w_{n,0}\ge Q_N^{-3/8}.
\]
Extension at $s=4$, or its dual restriction estimate at $p=4/3$,
would imply $\mu(I)\lesssim |I|^{1/2}$ for every interval $I$.
Indeed, for $I=[x-r,x+r]$, test restriction with
$r e^{2\pi ixt}\psi(rt)$, where $|\widehat\psi|\ge1$ on $[-1,1]$;
its $L^{4/3}$ norm is $r^{1/4}\|\psi\|_{4/3}$.
For $I=[0,Q_N^{-1})$, the two bounds contradict each other since
$Q_N^{1/8}\to\infty$.
\end{proof}

At $s=2$, the condition $\widehat\mu\in L^2(\mathbb R)$ and the full
extension estimate are governed by different products. This explains the strict inequality
$s>2$ in the exact moment criterion.
\begin{proposition}\label{prop:quadratic-endpoint}
For an arbitrary weighted Moran measure, discrete Parseval gives
$G_{2,n}=R_n\sum_jw_{n,j}^2-1$. The following equivalences hold:
\begin{equation}
 \widehat\mu\in L^2(\R)
 \quad\Longleftrightarrow\quad
 \prod_n(1+G_{2,n})<\infty
 \quad\Longleftrightarrow\quad\sum_nG_{2,n}<\infty,
 \label{eq:scalar-quadratic-endpoint}
 \end{equation}
 \begin{equation}
 L^2(\mu)\to L^2(\R)\text{ extension holds}
 \quad\Longleftrightarrow\quad\prod_n C_n<\infty.
 \label{eq:operator-quadratic-endpoint}
\end{equation}
In the second case, the square of the optimal extension constant is
$\|\frac{d\mu}{dx}\|_\infty=\prod_nC_n$.
\end{proposition}
\begin{proof}
Use the densities $f_N$ and cylinder partitions defined in the proof of
Example~\ref{ex:unbounded-frame-counterexample}, with the arbitrary
digit laws of the present proposition. The cylinder formula gives
\[
 \|f_N\|_2^2=\prod_{n\le N}(1+G_{2,n}),\qquad
 \|f_N\|_\infty=\prod_{n\le N}C_n.
\]
Boundedness of the first product gives an $L^2$ martingale limit;
boundedness of the second gives a bounded one. Since $f_N\,dx\Rightarrow\mu$,
either limit is the density of $\mu$. Conversely, if $\mu=f\,dx$, the
cylinder endpoints are $\mu$-null, so $f_N$ is the conditional expectation
of $f$ on the level-$N$ partition. Conditional expectation is contractive in $L^2$
and $L^\infty$, and $f_N\to f$ almost everywhere, giving the converse
implications and the asserted norm identities. Plancherel identifies
$\widehat\mu\in L^2$ with an $L^2$ density. The full $L^2$ extension
estimate is equivalent to a bounded density. Indeed, if $\mu=f\,dx$
with $f\in L^\infty$, put $M=\|f\|_\infty$. Plancherel gives
\[
 \|\widehat{g\mu}\|_2^2
 =\int |g|^2f^2\,dx
 \le M\|g\|_{L^2(\mu)}^2.
\]
Conversely, the dual restriction estimate forces
$\mu(I)\lesssim |I|$ by smooth interval tests, and hence $\mu$ has
a bounded density.

To identify the optimal constant in the bounded-density case, let
$0<\varepsilon<M$. The set
$E_\varepsilon=\{x:f(x)>M-\varepsilon\}$ has positive
$\mu$-measure. Taking $g=\one_{E_\varepsilon}$ gives
\[
 \frac{\|\widehat{g\mu}\|_2^2}{\|g\|_{L^2(\mu)}^2}
 =\frac{\int_{E_\varepsilon}f^2\,dx}
        {\int_{E_\varepsilon}f\,dx}
 \ge M-\varepsilon.
\]
Together with the upper bound and $\varepsilon\downarrow0$, this
shows that the square of the optimal extension constant equals
$M=\|f\|_\infty=\prod_nC_n$.
Finally, for nonnegative $G_{2,n}$, convergence of the product is
equivalent to convergence of their sum.
\end{proof}

For example, take $R_n=2$ and
$w_{n,0}=(1+a_n)/2$, $w_{n,1}=(1-a_n)/2$, where $a_n=1/(n+1)$.
Then $C_*\le3/2$ and $\sum_nG_{2,n}=\sum_na_n^2<\infty$, but
\[
 \prod_{n=1}^NC_n=\prod_{n=1}^N(1+a_n)=\frac{N+2}{2}\to\infty.
\]
By \eqref{eq:scalar-quadratic-endpoint} and
\eqref{eq:operator-quadratic-endpoint}, this example shows that
even when $C_*<\infty$, the condition $\widehat\mu\in L^2(\mathbb R)$
does not imply the $L^2(\mu)\to L^2(\mathbb R)$ extension estimate.

\section{Same Fourier bounds, different restriction exponents}\label{sec:fixed-envelope}

We apply the moment criterion while keeping the scales, digit cardinalities,
maximal mask magnitudes at nonzero grid frequencies and an exponential spectrum fixed. The resulting measures also satisfy a common Fourier decay bound, with a constant independent of the measure, and this bound is sharp in order along the same resonant sequence.

\subsection{Dimension and spectral notation}\label{sec:dimension-notation}

For a Borel probability measure $\mu$ on
$\R$, its \emph{Fourier dimension} is
\[
\dimF\mu=\sup\{0\le t\le1:
|\widehat\mu(\xi)|\lesssim_t(1+|\xi|)^{-t/2}\},
\]
while its \emph{correlation dimension} is
\[
\text{$\dimtwo\mu=\liminf_{r\downarrow0}\frac{\log\Corr_\mu(r)}{\log r},$ ~where ~
$
\Corr_\mu(r)=\iint\one_{\{|x-y|\le r\}}\dd\mu(x)\dd\mu(y).$}
\]
We write $\dimH\mu=\inf\{\dimH E:\mu(E)=1\}$ and use natural logarithms.
We use $-\log0=+\infty$ throughout. The standard energy comparison gives
$\dimF\mu\le\dimtwo\mu\le\dimH\mu$; see \cite{Mattila}.

To keep track of the mass distribution at all scales, we use the following
fixed-grid formulation of the \emph{lower $L^q$-dimension}. For $r>0$, let
$
\mathcal D_r=\{[kr,(k+1)r):k\in\Z\}
$
and for $1<q<\infty$ set
\[
\text{$D_q(\mu)
=
\liminf_{r\downarrow0}
\frac{-\log\mathcal M_q(\mu,r)}
{(q-1)\log(1/r)}$ \quad where ~~$\mathcal M_q(\mu,r)
=
\sum_{I\in\mathcal D_r}\mu(I)^q$.}
\]
We remark that replacing $\mathcal D_r$ by any interval partition whose interval lengths are uniformly comparable to $r$ changes $\mathcal M_q(\mu,r)$ only by a multiplicative constant independent of $r$, and hence does not affect $D_q(\mu)$. Thus this is the usual lower $L^q$-dimension; in particular, $D_2(\mu)=\dimtwo\mu$, the correlation dimension, see, for example, \cite{MattilaMoranRey}. We also use the \emph{Frostman dimension}
\[
D_\infty(\mu)=\sup\left\{a\ge0:\sup_x\mu(B(x,r))\lesssim_a r^a
\text{ for }0<r<1\right\}.
\]
Moreover, $D_q(\mu)\downarrow D_\infty(\mu)$ as $q\to\infty$, see, for example, \cite{FanLauRao}. For general background on \(L^q\)-dimensions and multifractal dimensions of measures, see also \cite{Olsen1995}.

The Fourier spectrum, introduced by Fraser \cite{FraserSpectrum},
interpolates between uniform Fourier decay and Sobolev regularity.
It records weighted integrability of $\widehat\mu$ and yields a family of
Stein--Tomas type restriction estimates \cite{CFdORestriction}.

For a finite Borel measure $\mu$ on $\mathbb R$, $t\ge0$, and
$0<\theta\le1$, set
\[
\mathcal J_{t,\theta}(\mu)
:=
\left(
\int_{\mathbb R}
|\widehat\mu(\xi)|^{2/\theta}
|\xi|^{t/\theta-1}\dd\xi
\right)^{\theta},
\]
and define the endpoint quantity
$
\mathcal J_{t,0}(\mu)
:=
\sup_{\xi\in\mathbb R}
|\widehat\mu(\xi)|^2|\xi|^t.
$
Following Fraser \cite{FraserSpectrum}, the \emph{Fourier spectrum} of $\mu$ is the function
$\theta\mapsto\dim_{\mathrm F}^{\theta}\mu$, $0\le\theta\le1$, where
\begin{equation}\label{eq:Fourier-spectrum-definition}
\dim_{\mathrm F}^{\theta}\mu
:=
\sup\{t\ge0:\mathcal J_{t,\theta}(\mu)<\infty\}.
\end{equation}
We adopt the convention $\sup\varnothing=0$.
Thus the endpoint $\theta=0$ recovers the usual Fourier dimension after truncation at the ambient dimension:
$
\dimF\mu=\min\{1,\dim_{\mathrm F}^{0}\mu\},
$
whereas $\dim_{\mathrm F}^{1}\mu=\dimS\mu$ is the Sobolev dimension, see \cite{Mattila}.

A Borel probability measure $\mu$ is \emph{spectral} if there is a
countable set $\Lambda\subset\R$ such that
$\{e_\lambda:\lambda\in\Lambda\}$, where
$e_\lambda(x):=e^{2\pi i\lambda x}$, is an orthonormal basis for
$L^2(\mu)$. We call such a set $\Lambda$ an \emph{exponential spectrum} of $\mu$ (or simply a \emph{spectrum} when no confusion can arise).
Hadamard triples and uniform bounds on normalized tails provide the
spectrality criterion used below; see \cite{JorgensenPedersen,LaiWang}.

For a finite digit set $B\subset\{0,\ldots,R-1\}$, we say that $B$ is \emph{Hadamard} (for the base $R$) if there exists a finite set $L\subset\Z$ with $|L|=|B|$ such that
\[
\frac{1}{\sqrt{|B|}}\bigl(e^{2\pi i b\ell/R}\bigr)_{b\in B,\,\ell\in L}
\]
is unitary; equivalently, $(R,B,L)$ is a Hadamard triple.

For the measures constructed below, $D_\infty=1$ implies $D_q=1$ for
all $q>1$, while Fourier dimension zero and Sobolev dimension one imply
$\dim_{\mathrm F}^{\theta}\mu=\theta$ by
\cite[Corollary~1.6]{FraserSpectrum}. In this boundary case the dimension
function does not say whether the defining integral is finite at its
critical exponent. Main Theorem A identifies that information exactly:
\begin{equation}\label{eq:critical-line-moments}
 \mathcal J_{\theta,\theta}(\mu)<\infty
 \quad\Longleftrightarrow\quad
 \sum_nG_{2/\theta,n}<\infty,
 \qquad 0<\theta<1,
\end{equation}
whenever $C_*<\infty$. Indeed,
\[
 \mathcal J_{\theta,\theta}(\mu)
 =\|\widehat\mu\|_{L^{2/\theta}(\mathbb R)}^2,
 \qquad 0<\theta<1,
\]
and the equivalence follows from Theorem~\ref{thm:exact-moment-criterion}.

\subsection{Common scales, maxima and exponential spectrum}

We now construct measures with common Fourier bounds and different restriction exponents.

\begin{theorem}\label{thm:fixed-envelope}
There are fixed integer sequences $R_n=2M_n$, with $M_n$ odd, and a
fixed positive sequence $\eta_n\asymp n^{-1/2}$, with the following
property. For every $p_0\in[4/3,2]$ there are digit sets
$B_n^{(p_0)}\subset\{0,\ldots,R_n-1\}$ of cardinality $M_n$ such that
\[
 \mu_{p_0}=\bigconv_{n=1}^\infty
 \left(\frac1{M_n}\sum_{b\in B_n^{(p_0)}}\delta_{b/Q_n}\right)
\]
is singular continuous and spectral, and $p_{\rm res}(\mu_{p_0})=p_0$.
More precisely, the following assertions hold simultaneously.
\begin{enumerate}
\item [\textup{(i)}] All these measures have the same exponential spectrum $\Lambda$. Their
one-level frame bounds are $C_n=2$, and their maxima are exactly equal:
\begin{equation*}
 \max_{1\le k<R_n}|m_n^{(p_0)}(k/R_n)|=
 m_n^{(p_0)}(1/2)=\eta_n
 \quad(n\ge1).
\end{equation*}
In addition, $G_{2,n}=1$ for every member of the family.
\item [\textup{(ii)}] A constant independent of $p_0$ gives the common envelope
\begin{equation}\label{eq:fixenv-envelope}
 |\widehat\mu_{p_0}(\xi)|\le
 C\frac{\log\log(e^e+|\xi|)}{\sqrt{\log(e+|\xi|)}}.
\end{equation}
At the same frequencies $\xi_n=Q_n/2$, every member satisfies
\begin{equation}\label{eq:fixenv-resonance}
 |\widehat\mu_{p_0}(\xi_n)|
 =\left(\frac2\pi+o(1)\right)\eta_n
 \asymp\frac{\log\log\xi_n}{\sqrt{\log\xi_n}}.
\end{equation}
The error term is uniform over the family.
\item [\textup{(iii)}] For every member,
\begin{equation}\label{eq:fixenv-dimensions}
 D_q(\mu_{p_0})=1\quad(1<q\le\infty),\qquad
 \dim_{\mathrm F}^{\theta}\mu_{p_0}=\theta\quad(0\le\theta\le1),
\end{equation}
and the pointwise local dimension is one at every point of its support.
\item [\textup{(iv)}] If $p_0<2$, restriction holds exactly for $1\le p\le p_0$.
For $p_0=2$, it holds exactly for $1\le p<2$.
\end{enumerate}
\end{theorem}

The fixed data in this theorem are chosen once for the entire family;
they are not arbitrary data prescribed in advance. The construction belongs to the class of spectral Moran measures with equal weights at each level. It does not impose the additional
prime-power parametrization and finite-field digit formula used in the
finite-field examples in Section~\ref{sec:restriction}. Its finite sign vectors are selected
by an existence argument and then fixed, so the resulting measures are
deterministic.

To prove Theorem~\ref{thm:fixed-envelope}, we first construct sign vectors
with prescribed mean and controlled Fourier coefficients.
Later, periodic repetition preserves the nonzero Fourier coefficients
of a short sign vector while embedding them in a larger grid.

The following is a standard probabilistic lemma, obtained from
Hoeffding's inequality \cite[Theorem~2]{Hoeffding} and a simple
adjustment enforcing the prescribed sum. We include a proof for
completeness.

\begin{lemma}\label{lem:fixed-envelope-signs}
There is an absolute constant $A$ such that, for any odd integers $m\ge3$
and $a$ with $|a|\le m$, there is a sign vector
$\sigma=(\sigma_j)_{j=0}^{m-1}\in\{-1,1\}^m$ with
\begin{equation}\label{eq:sign-mean-bound}
 \bar c:=\frac1m\sum_{j=0}^{m-1}\sigma_j=\frac am,
 \qquad
 \max_{1\le k<m}\left|\frac1m\sum_{j=0}^{m-1}\sigma_je^{-2\pi ikj/m}\right|
 \le\delta(m):=A\sqrt{\frac{\log(2m)}m}.
\end{equation}
\end{lemma}

\begin{proof}
Choose independent signs $X_j$ with $\mathbb E X_j=\bar c=a/m$. For each $k$,
apply Hoeffding's inequality to the real and imaginary
parts of $\sum_{j=0}^{m-1}(X_j-\bar c)e^{-2\pi ikj/m}$. It gives
\[
 \mathbb P\!\left(\left|\sum_{j=0}^{m-1}(X_j-\bar c)e^{-2\pi ikj/m}\right|>v\right)
 \le4\exp(-v^2/(4m)).
\]
A union bound over $k=0,\ldots,m-1$ shows that, for a sufficiently large
absolute $A_1$, there is a realization in which all these sums have modulus
at most $A_1\sqrt{m\log(2m)}$. Its total sum differs from $a$ by an even
integer of this size. Flip exactly $|\sum_{j=0}^{m-1}X_j-a|/2$ signs of the
appropriate type. This is possible because $a\in[-m,m]$ and has the
correct parity. The new total is exactly $a$, and each Fourier sum changes
by at most $|\sum_{j=0}^{m-1}X_j-a|$. For $k\ne0$ the original expected Fourier
sum was zero. Increasing the absolute constant proves the lemma.
\end{proof}

\subsection{Spectral digit sets and the common construction}

For an odd integer $m\ge3$ and a sign vector
$\sigma=(\sigma_j)_{j=0}^{m-1}\in\{-1,1\}^m$, define
\[
 b_j(\sigma)=
 \begin{cases}
 j,&\sigma_j=(-1)^j,\\
 j+m,&\sigma_j=-(-1)^j,
 \end{cases}
 \qquad 0\le j<m,
\]
and set $B(\sigma)=\{b_j(\sigma):0\le j<m\}\subset\{0,\ldots,2m-1\}$.
Then $(-1)^{b_j(\sigma)}=\sigma_j$ and
$b_j(\sigma)\equiv j\pmod m$ for every $j$; in particular,
$|B(\sigma)|=m$.

Our digit construction is related in spirit to those of Li, Miao and Wang
\cite[Example~1.6 and Section~5]{LiMiaoWang}.
Both modify digit representatives while preserving the finite Fourier matrix.
Their large shifts produce spectral measures with unbounded support.
Here the choices $j$ and $j+m$ remain in $\{0,\ldots,2m-1\}$ and encode
sign vectors that control the nonzero Fourier moments.

Write $B=B(\sigma)$, with $\bar c$
as in \eqref{eq:sign-mean-bound}, and set
\[
 F_\sigma(k)=\frac1m\sum_{j=0}^{m-1}\sigma_je^{-2\pi ikj/m},\qquad
 m_B(u)=\frac1m\sum_{b\in B(\sigma)}e^{-2\pi ibu}.
\]
The even and odd grid frequencies give the exact identities
\begin{align}
 m_B(2l/(2m))&=0&& (1\le l<m),\label{eq:fixenv-even}\\
 m_B((m+2l)/(2m))&=F_\sigma(l)&& (l\bmod m).
 \label{eq:fixenv-odd}
\end{align}
Frequencies in \eqref{eq:fixenv-odd} are understood modulo $2m$.
In particular $m_B(1/2)=F_\sigma(0)=\bar c$.
Moreover,
\begin{equation}\label{eq:fixenv-finite-hadamard}
 \left(2m,B(\sigma),
 2\{-(m-1)/2,\ldots,(m-1)/2\}\right)
 \quad\text{is a Hadamard triple}.
\end{equation}
Indeed, for $0\le j<m$ and $|\ell|\le(m-1)/2$,
\[
 \frac1{\sqrt m}e^{2\pi i b_j(\sigma)(2\ell)/(2m)}
 =\frac1{\sqrt m}e^{2\pi i j\ell/m}.
\]
The indices $j$ and $\ell$ each run through all residues modulo $m$,
so these entries form the unitary Fourier matrix on $\Z/m\Z$.

For a nonempty digit set $B\subset\{0,\ldots,R-1\}$ with a specified
integer base $R\ge2$, write
$m_B(u)=|B|^{-1}\sum_{b\in B}e^{-2\pi ibu}$ and define
\[
 G_s(B):=\sum_{k=1}^{R-1}|m_B(k/R)|^s\quad(s>2),
 \qquad
 \eta(B):=\max_{1\le k<R}|m_B(k/R)|.
\]
The base $R$ is determined by the construction of $B$ and is suppressed
in this notation. In particular, $R=2m$ for $B(\sigma)$, whereas
$R=2M$ for the repeated set $B(\widetilde\sigma)$ introduced below.

For $B=B(\sigma)$, suppose $0<\bar c<1$ and that $\sigma$ satisfies
the Fourier bound in \eqref{eq:sign-mean-bound}.
Parseval for $F_\sigma$ gives $\sum_{k=0}^{m-1}|F_\sigma(k)|^2=1$, so
\begin{equation}\label{eq:fixenv-finite-moments}
 \begin{split}
 G_s(B)&=\bar c^s+H_s(\sigma),\\
 (m-1)^{1-s/2}(1-\bar c^2)^{s/2}
 &\le H_s(\sigma)
 :=\sum_{k=1}^{m-1}|F_\sigma(k)|^s
 \le\delta(m)^{s-2}(1-\bar c^2),~~ s>2.
 \end{split}
\end{equation}
The lower bound is the finite-dimensional $\ell^2$--$\ell^s$
inequality; the upper bound uses the maximum and Parseval.
If $\delta(m)<\bar c$, then the maximum in the full nonzero grid is
\emph{exactly} $\bar c$.

We will also need control between the grid points. For any trigonometric
polynomial
\[
 P(t)=\sum_{j=0}^{r-1}a_j e^{-2\pi ijt},\qquad a_j\in\mathbb C,
\]
Fourier interpolation gives
\[
 P(t)=\sum_{k=0}^{r-1}P(k/r)D_r(t-k/r),\qquad
 \sup_t\sum_{k=0}^{r-1}|D_r(t-k/r)|\le C\log(2r).
\]
The last estimate follows from
$|D_r(u)|\le\min\{1,(2r\operatorname{dist}(u,\Z))^{-1}\}$ and the
harmonic-series bound after ordering the distances to the grid.
Apply interpolation to $m_B-D_{2m}$, keeping the coefficient at $1/2$
separate. Equations~\eqref{eq:fixenv-even}--\eqref{eq:fixenv-odd}, together
with \eqref{eq:sign-mean-bound}, yield
\begin{equation}\label{eq:fixenv-continuous-mask}
 \|m_B-D_{2m}\|_\infty
 \le\bar c+C\delta(m)\log(2m)
 \le\bar c+C\frac{(\log(2m))^{3/2}}{\sqrt m}.
\end{equation}

Finally, let $M=mK$ with both $m$ and $K$ odd, and repeat $\sigma$
periodically to a sign vector $\widetilde\sigma$ of length $M$.
For each $b\in B(\sigma)$, the numbers $b+2mq$, $0\le q<K$,
represent all residue classes modulo $M$ congruent to $b$ modulo $m$:
multiplication by $2$ permutes the classes modulo the odd integer $K$.
Their parity is unchanged, so they select exactly the digits prescribed
by $\widetilde\sigma$. Thus
\begin{equation}\label{eq:fixenv-repeat}
 B(\widetilde\sigma)=B(\sigma)+2m\{0,\ldots,K-1\},\qquad
 m_{B(\widetilde\sigma)}(t)=m_{B(\sigma)}(t)D_K(2mt).
\end{equation}
Consequently,
\begin{equation}\label{eq:fixenv-repeat-properties}
 \begin{split}
 G_s(B(\widetilde\sigma))=G_s(B(\sigma)),\qquad
 \eta(B(\widetilde\sigma))=\eta(B(\sigma)),\\
 \|m_{B(\widetilde\sigma)}-D_{2M}\|_\infty
 \le\|m_{B(\sigma)}-D_{2m}\|_\infty.
 \end{split}
\end{equation}
For the first two identities, the length-$M$ Fourier transform of the
repeated vector vanishes outside the multiples of $K$ and equals
$F_\sigma$ at those multiples. The last inequality follows from
\eqref{eq:fixenv-repeat} and $D_{2M}(t)=D_{2m}(t)D_K(2mt)$ for $t\in \R$.

\begin{proof}[Proof of Theorem~\ref{thm:fixed-envelope}]
We now assemble the preceding ingredients to construct the required family.

\smallskip
\noindent\emph{Step 1: Common scales and two digit choices.}

Choose a sufficiently large fixed integer $n_0$, and put $u_n=n+n_0$.
Let $L_n$ and $K_n$ be odd integers satisfying
\begin{equation}\label{eq:fixenv-scales}
 \begin{split}
 u_n(\log u_n)^4&\le L_n\le u_n(\log u_n)^4+2,\\
 e^{(\log u_n)^2}&\le K_n\le e^{(\log u_n)^2}+2,\\
 M_n&=L_nK_n,\qquad R_n=2M_n.
 \end{split}
\end{equation}
Choose an odd integer $a_n$ with
$|a_n-L_nu_n^{-1/2}|\le1$, and set
\begin{equation}\label{eq:fixenv-eta-choice}
 t_n=\frac{a_n}{L_n}\asymp u_n^{-1/2}.
\end{equation}
Taking $n_0$ large ensures $0<t_n<1/2$ and
$\delta(L_n),\delta(M_n)<t_n/2$ for every $n$,
where $\delta(m)$ is defined in \eqref{eq:sign-mean-bound}.

At every level use Lemma~\ref{lem:fixed-envelope-signs} to choose two sign
vectors. The short one, $\sigma_n$, has length $L_n$ and sum $a_n$.
The long one, $\rho_n$, has length $M_n$ and sum $a_nK_n$.
The second application is legitimate since $a_nK_n$ is odd, and both
means equal $t_n$. Define
$
 B_n^{\rm r}=B(\widetilde\sigma_n)$ and $
 B_n^{\rm f}=B(\rho_n),$
where $\widetilde\sigma_n$ is the repetition of $\sigma_n$ to length
$M_n$. Using the common base $R_n=2M_n$, we set
\[
 G_{s,n}^{\rm r}:=G_s(B_n^{\rm r}),\qquad
 G_{s,n}^{\rm f}:=G_s(B_n^{\rm f}),\qquad s>2.
\]
We also define the uniform mask errors by
\[
 \varepsilon_n^{\rm r}:=\|m_{B_n^{\rm r}}-D_{R_n}\|_\infty,\qquad
 \varepsilon_n^{\rm f}:=\|m_{B_n^{\rm f}}-D_{R_n}\|_\infty,
\]
where the supremum may be taken over $t\in[0,1]$ by periodicity.
Both sets \(B_n^{\mathrm r}\) and \(B_n^{\mathrm f}\) have cardinality $M_n$, both have maximum exactly
$\eta_n=t_n$, and both have $G_{2,n}=1$.
Equations~\eqref{eq:fixenv-continuous-mask}--\eqref{eq:fixenv-repeat-properties} give
\begin{equation}\label{eq:fixenv-uniform-eps}
 \max\{\varepsilon_n^{\rm r},\varepsilon_n^{\rm f}\}\le C u_n^{-1/2}.
\end{equation}
Here the short-vector error is at most
$C u_n^{-1/2}(\log u_n)^{-1/2}$ in addition to $t_n$, and the
long-vector error is smaller. Since $C_n=2$, the exact moment criterion
applies to every choice of these two sets at successive levels.

\smallskip
\noindent\emph{Step 2: Prescribing the restriction exponent.}

Fix $s>2$, and abbreviate $\gamma=(s-2)/2$. For the long vector,
\begin{equation}\label{eq:fixenv-flat-summable}
 G_{s,n}^{\rm f}\le t_n^s+C_s
 \left(\frac{\log(2M_n)}{M_n}\right)^\gamma,
 \qquad \sum_nG_{s,n}^{\rm f}<\infty.
\end{equation}
Indeed, $\sum_nt_n^s<\infty$ and $M_n$ grows faster than every fixed
power of $n$. For the repeated vector, \eqref{eq:fixenv-finite-moments} gives
\begin{equation}\label{eq:fixenv-repeated-bounds}
 \frac{c_s u_n^{-\gamma}}{(\log u_n)^{4\gamma}}
 \le G_{s,n}^{\rm r}-t_n^s
 \le\frac{C_s u_n^{-\gamma}}{(\log u_n)^{3\gamma}}.
\end{equation}

For $0<d\le1$, define the set of levels
\begin{equation}\label{eq:fixenv-selected-levels}
 \mathcal A_d=\left\{
 n_j(d):=\left\lceil
 \bigl(j(\log(j+e))^2\bigr)^{1/d}\right\rceil:j\ge1\right\},
 \qquad \mathcal A_0=\varnothing.
\end{equation}
These indices are strictly increasing. At levels in $\mathcal A_d$ use
$B_n^{\rm r}$; at all other levels use $B_n^{\rm f}$. Denote the resulting
equal-weighted Moran measure by $\mu_d$.

For $d>0$, equations~\eqref{eq:fixenv-flat-summable}--\eqref{eq:fixenv-repeated-bounds}
show
\begin{equation}\label{eq:fixenv-threshold-series}
 \sum_nG_{s,n}(\mu_d)<\infty
 \quad\Longleftrightarrow\quad s\ge2+2d,
 \qquad s>2.
\end{equation}
To check the endpoint, if $\gamma=d$ then the upper bound on the selected
levels is summable, since
\[
 u_{n_j(d)}^{-d}(\log u_{n_j(d)})^{-3d}
 \lesssim_d\frac1{j(\log(j+e))^2}.
\]
For $\gamma>d$, summability follows from the power of $j$ alone.
If $0<\gamma<d$, the lower bound in \eqref{eq:fixenv-repeated-bounds}, evaluated
at $n_j(d)$, is comparable to
$
 j^{-\gamma/d}(\log j)^{-2\gamma/d-4\gamma},
$
whose series diverges. This proves \eqref{eq:fixenv-threshold-series} without
needing an exact asymptotic for the individual sign-vector moments.
For $d=0$, the series converges for every $s>2$.

Apply Theorem~\ref{thm:exact-moment-criterion} and duality. We obtain
\begin{equation*}
 p_{\rm res}(\mu_d)=\frac{2+2d}{1+2d},\qquad 0\le d\le1.
\end{equation*}
For $d>0$ the critical endpoint holds by \eqref{eq:fixenv-threshold-series}.
Given $p_0\in[4/3,2]$, take
$d=(2-p_0)/(2(p_0-1))$. It remains to verify spectrality, singularity,
the common envelope, and the dimension assertions.

\smallskip
\noindent\emph{Step 3: A common exponential spectrum and singular continuity.}
Let $\Lambda=\bigcup_N\Lambda_N,$ where
\[
 \text{$\mathcal L_n=2\{-(M_n-1)/2,\ldots,(M_n-1)/2\}$~~ and~~ $
 \Lambda_N=\mathcal L_1+R_1\mathcal L_2+\cdots+Q_{N-1}\mathcal L_N.$}
\]
Every finite prefix has exponential spectrum $\Lambda_N$, by
\eqref{eq:fixenv-finite-hadamard}. These sets are independent of $d$, are nested,
and satisfy
$|\lambda|\le\frac{Q_N-1}{2}$ with $\lambda\in\Lambda_N.$

For each member of the family, let $\nu_N$ be the normalized tail from
Section~\ref{sec:criterion}. The uniform mask approximation gives,
for $|t|\le1/2$,
\begin{equation}\label{eq:fixenv-tail-uniform}
 |\widehat\nu_N(t)-\widehat\unif(t)|
 \le\varepsilon_{N+1}+\frac{4\pi|t|}{R_{N+1}}\longrightarrow0.
\end{equation}
Indeed, factor the two transforms as
\[
 m_{N+1}(t/R_{N+1})\widehat\nu_{N+1}(t/R_{N+1}),\qquad
 D_{R_{N+1}}(t/R_{N+1})\widehat\unif(t/R_{N+1}),
\]
and use $|\widehat\rho(u)-1|\le2\pi|u|$ for probabilities supported
on $[0,1]$. Since $|\widehat\unif(t)|\ge2/\pi$ on this interval,
\eqref{eq:fixenv-tail-uniform} gives a common positive lower bound for the
tail factors at $\lambda/Q_N$, $\lambda\in\Lambda_N$, for large $N$.

Digit expansions are unique outside cylinder endpoints, and each endpoint
has measure zero because $(M_1\cdots M_N)^{-1}\to0$.
The preceding tail bound is the usual completeness condition for Hadamard Moran measures
(see, for example, \cite[Theorem~1.3]{LaiWang}). Here is a direct verification.
The exponential family indexed by $\Lambda$ is orthonormal. If $g$ is
a function of the first $N$ digits, let $\mu_{\le N}$ be the finite-prefix
measure and write
\[
 X=X_{\le N}+\frac{Z_N}{Q_N},
\]
where $X_{\le N}$ and $Z_N$ are independent and $Z_N$ has law $\nu_N$.
Then, for $\lambda\in\Lambda_N$,
\begin{align*}
 \langle g,e_\lambda\rangle_{L^2(\mu_d)}
 =\mathbb E\left[g(X_{\le N})e^{-2\pi i\lambda X_{\le N}}
 e^{-2\pi i(\lambda/Q_N)Z_N}\right]=\langle g,e_\lambda\rangle_{L^2(\mu_{\le N})}
 \widehat\nu_N(\lambda/Q_N).
\end{align*}
For all sufficiently large $N$, finite-prefix Parseval and the tail lower bound give
\[
 \sum_{\lambda\in\Lambda_N}|\langle g,e_\lambda\rangle_{L^2(\mu_d)}|^2
 \ge c^2\|g\|_{L^2(\mu_d)}^2.
\]
Let $P$ be the orthogonal projection onto the closed span of
$\{e_\lambda:\lambda\in\Lambda\}$. The preceding inequality implies
$\|Pg\|_2\ge c\|g\|_2$ on the dense class of functions depending on
finitely many digits, hence on all of $L^2(\mu_d)$. In particular, if $h\in\ker P$, then
$0=\|Ph\|_2\ge c\|h\|_2$, so $\ker P=\{0\}$. Since an orthogonal
projection has range $(\ker P)^\perp$, it follows that $P$ is the identity.
This proves that $\Lambda$ is a common exponential spectrum.

At level $N$, every nonempty cylinder has mass
\begin{equation}\label{eq:fixenv-cylinder-mass}
 \frac1{M_1\cdots M_N}=\frac{2^N}{Q_N}.
\end{equation}
The covering by these cylinders has total length $2^{-N}$, while their
masses tend to zero. Hence each measure is singular and continuous.
In particular, restriction at $p=2$ fails by
Proposition~\ref{prop:quadratic-endpoint}, since $C_n=2$ for every $n$.
This also completes the endpoint assertion for $d=0$.

\smallskip
\noindent\emph{Step 4: Common Fourier decay and common resonances.}
Here $\varepsilon_n=O(n^{-1/2})$, so the mask errors need not be summable.
We retain a small-frequency factor in the telescoping estimate. If
$\mathcal E_n=m_n-D_{R_n}$, then $\mathcal E_n(0)=0$, and Bernstein's inequality gives
$
 \|\mathcal E_n'\|_\infty\le2\pi(R_n-1)\varepsilon_n.
$
Let $\mu^{[n]}$ be the first $n$ digit levels followed by a uniform tail,
and put $\mu^{[0]}=\unif$. Factoring consecutive transforms yields
\begin{equation*}
 |\widehat{\mu^{[n]}}(\xi)-\widehat{\mu^{[n-1]}}(\xi)|
 \le C\varepsilon_n
 \min\left\{1,\frac{|\xi|}{Q_{n-1}}\right\}
 \min\left\{1,\frac{Q_n}{|\xi|}\right\}.
\end{equation*}
Here the first minimum uses $\mathcal E_n(0)=0$ and Bernstein, and the second
uses the Fourier transform of the uniform tail. For a fixed $\xi$ this
series is absolutely convergent, so telescoping and weak convergence give
\begin{equation}\label{eq:fixenv-refined-envelope}
 |\widehat\mu_d(\xi)|\le\frac C{|\xi|}
 +C\sum_{n\ge1}\varepsilon_n
 \min\left\{1,\frac{|\xi|}{Q_{n-1}}\right\}
 \min\left\{1,\frac{Q_n}{|\xi|}\right\},\quad |\xi|\ge1.
\end{equation}
For $Q_{N-1}\le|\xi|<Q_N$, use $R_j\ge2$ and
\eqref{eq:fixenv-uniform-eps}. The terms before $N$ are bounded by
$C\sum_{n<N}u_n^{-1/2}2^{-(N-1-n)}$, and those after $N$ by
$C\sum_{n>N}u_n^{-1/2}2^{-(n-N-1)}$. Both are $O(u_N^{-1/2})$.
Consequently
\begin{equation}\label{eq:fixenv-level-envelope}
 |\widehat\mu_d(\xi)|\le C u_N^{-1/2}
 \qquad(Q_{N-1}\le|\xi|<Q_N).
\end{equation}
By \eqref{eq:fixenv-scales},
\begin{equation}\label{eq:fixenv-scale-growth}
 T_N:=\log Q_N\sim N(\log N)^2,\quad
 \frac{T_N}{T_{N-1}}\rightarrow1,\quad \frac{N}{T_N}\longrightarrow0.
\end{equation}
Converting \eqref{eq:fixenv-level-envelope} from $N$ to $|\xi|$ proves
\eqref{eq:fixenv-envelope}.

At $\xi_n=Q_n/2$, all preceding masks equal one, since
$Q_n/(2Q_j)$ is an integer for $j<n$. The $n$th mask equals $t_n$.
Thus the exact identity
$
 \widehat\mu_d(Q_n/2)=t_n\widehat\nu_n(1/2)
$
and \eqref{eq:fixenv-tail-uniform} imply \eqref{eq:fixenv-resonance}. This proves
that the common envelope is sharp in order along a common sequence.

\smallskip
\noindent\emph{Step 5: Common dimension data.}
Choose $N$ so that $Q_N^{-1}\le r<Q_{N-1}^{-1}$. For every point
$x\in\supp\mu_d$, a level-$N$ cylinder containing $x$ lies in
$B(x,r)$, up to zero-mass endpoints. On the other hand, a ball of radius
$r$ meets at most $2rQ_N+3$ grid intervals at that level. By
\eqref{eq:fixenv-cylinder-mass},
\begin{equation}\label{eq:fixenv-ball-masses}
 \frac{2^N}{Q_N}\le\mu_d(B(x,r))\le5r2^N.
\end{equation}
The upper bound holds for all centers in $\R$. Set $L=-\log r$;
then $T_{N-1}<L\le T_N$. Equations~\eqref{eq:fixenv-scale-growth} and
\eqref{eq:fixenv-ball-masses} give, uniformly on the support,
\[
 1-\frac{N\log2+\log5}{L}
 \le\frac{-\log\mu_d(B(x,r))}{L}
 \le\frac{T_N-N\log2}{T_{N-1}}\longrightarrow1.
\]
Thus every point has local dimension one. The uniform upper mass estimate
also gives $\mu_d(B(x,r))\le C_a r^a$ for every $a<1$. Therefore
$D_\infty=1$, and monotonicity and the ambient upper bound imply
$D_q=1$ for all $1<q\le\infty$. The support has Hausdorff dimension one.

The resonances \eqref{eq:fixenv-resonance} decay more slowly than any positive
power of $\xi_n$, so $\dim_{\mathrm F}^0\mu_d=0$. To verify the entire Fourier
spectrum directly, recall that for $0<\theta\le1$ its critical exponent is
determined by
$
 \int_{\R}|\widehat\mu_d(\xi)|^{2/\theta}
 |\xi|^{a/\theta-1}\,\dd\xi.
$
If $0<a<\theta$, this integral is finite: use
$|\widehat\mu_d|^{2/\theta}\le|\widehat\mu_d|^2$ and finite
$a/\theta$-energy, which follows from the preceding mass bounds.
Hence $\dim_{\mathrm F}^\theta\mu_d\ge\theta$.

Conversely, $\supp\mu_d\subset[0,1]$ implies
$|\widehat\mu_d'(\xi)|\le2\pi$. Around each $\xi_n$ there is consequently
an interval of length $c t_n$ on which
$|\widehat\mu_d|\ge c t_n$, with absolute $c>0$. If $a>\theta$, its
contribution to the displayed integral is at least
$
 c_{a,\theta}\,t_n^{2/\theta+1}\xi_n^{a/\theta-1}\longrightarrow\infty,
$
by \eqref{eq:fixenv-eta-choice} and \eqref{eq:fixenv-scale-growth}. The integral is
therefore infinite, proving $\dim_{\mathrm F}^\theta\mu_d\le\theta$. This completes
\eqref{eq:fixenv-dimensions} and the proof of Theorem~\ref{thm:fixed-envelope}.
\end{proof}

\subsection{Optimal restriction ranges and endpoint behavior
for fixed maximal Fourier magnitudes}

The following proposition bounds the critical restriction exponent
in terms of the summability of $(\eta_n)$.

\begin{proposition}\label{prop:fixed-envelope-obstruction}
Let $s_c=s_c(\mu)$ be as in
\eqref{eq:critical-moment-exponent}.
Assume $C_*<\infty$, and set
$
 \tau=\inf\{t>0:\sum_n\eta_n^t<\infty\}\in[0,\infty].
$
If $\tau<\infty$, then
\begin{equation}\label{eq:fixenv-universal}
 \max\{2,\tau\}\le s_c\le\tau+2.
\end{equation}
If $\tau=\infty$, then $s_c=\infty$ and $p_{\rm res}=1$.
\end{proposition}

\begin{proof}
Discrete Parseval gives
$
 G_{2,n}=R_n\sum_bw_{n,b}^2-1\le C_n-1.
$
Using $C_*$ from \eqref{eq:one-level-quantities}, we obtain,
for every $s>2$,
\begin{equation*}
 \eta_n^s\le G_{s,n}
 \le\eta_n^{s-2}G_{2,n}\le(C_*-1)\eta_n^{s-2}.
\end{equation*}
The lower estimate rules out summability for $2<s<\tau$, whereas
the upper estimate gives it for $s>\tau+2$. This proves
\eqref{eq:fixenv-universal}, including the asserted conclusions at
$\tau=0,\infty$.
\end{proof}

For $\eta_n\asymp n^{-1/2}$,
Proposition~\ref{prop:fixed-envelope-obstruction} gives
$p_{\rm res}\in[4/3,2]$. Thus the interval in
Theorem~\ref{thm:fixed-envelope} is sharp at the level of critical
restriction exponents for its fixed sequence $(\eta_n)$.

The construction in the proof of Theorem~\ref{thm:fixed-envelope} extends to every finite positive summability index
$\tau$. This shows that the universal bounds in
Proposition~\ref{prop:fixed-envelope-obstruction} cannot be improved, even after requiring
common scales, digit cardinalities, an exponential spectrum, and a sharp Fourier envelope.

\begin{theorem}\label{thm:fixed-envelope-general-tau}
Fix $0<\tau<\infty$, and put
\[
 a_\tau=\max\{2,\tau\},\qquad
 I_\tau=\left[\frac{\tau+2}{\tau+1},
                   \frac{a_\tau}{a_\tau-1}\right].
\]
There are common sequences $R_n=2M_n$, $\eta_n\asymp n^{-1/\tau}$,
and a common exponential spectrum $\Lambda$, such that every $p_0\in I_\tau$ is
realized as $p_{\rm res}(\mu_{\tau,p_0})$ by a singular continuous equal-weighted
spectral Moran measure with those data. All these measures satisfy
\eqref{eq:fixenv-dimensions} and have local dimension one at every support point.
They have the common quantitative bound
\begin{equation}\label{eq:fixenv-tau-envelope}
 |\widehat\mu_{\tau,p_0}(\xi)|\le
 C_\tau\left(
 \frac{[\log\log(e^e+|\xi|)]^2}{\log(e+|\xi|)}
 \right)^{1/\tau},
\end{equation}
and, uniformly over $p_0\in I_\tau$,
$
 |\widehat\mu_{\tau,p_0}(Q_n/2)|
 =\left(\frac2\pi+o(1)\right)\eta_n.
$
For $p_0<a_\tau/(a_\tau-1)$ the critical endpoint holds, whereas for
$p_0=a_\tau/(a_\tau-1)$ it fails. The interval $I_\tau$ is the full
possible interval of critical restriction exponents among weighted Moran
measures with $C_*<\infty$ and $\eta_n\asymp n^{-1/\tau}$.
\end{theorem}

\begin{proof}
We use Lemma~\ref{lem:fixed-envelope-signs} and the same parity
construction. Replace only the short lengths and their prescribed means in
\eqref{eq:fixenv-scales}--\eqref{eq:fixenv-eta-choice}: choose odd integers $L_n,K_n,a_n$
so that
\begin{equation*}
 \begin{aligned}
 u_n^{2/\tau}(\log u_n)^4&\le L_n
       \le u_n^{2/\tau}(\log u_n)^4+2,\\
 e^{(\log u_n)^2}&\le K_n\le e^{(\log u_n)^2}+2,\\
 |a_n-L_nu_n^{-1/\tau}|&\le1,\qquad
 t_n=\frac{a_n}{L_n},\quad M_n=L_nK_n,\quad R_n=2M_n.
 \end{aligned}
\end{equation*}
The shift $n_0$ may depend on $\tau$. By \eqref{eq:sign-mean-bound},
the short and long vectors have means exactly $t_n$, and all their other
coefficients are smaller than $t_n/2$ once $n_0$ is large.
Hence both digit sets have the exact same
maximum $\eta_n=t_n\asymp n^{-1/\tau}$, and
$
 \max\{\varepsilon_n^{\rm r},\varepsilon_n^{\rm f}\}
 \le C_\tau u_n^{-1/\tau}\rightarrow0.
$

For $s>2$, the long-vector contribution outside the mean is summable:
\begin{equation*}
 \sum_n\bigl(G_{s,n}^{\rm f}-t_n^s\bigr)<\infty,
 \qquad
 \sum_nt_n^s<\infty\ \Longleftrightarrow\ s>\tau.
\end{equation*}
For the repeated vector, \eqref{eq:fixenv-finite-moments} now gives
\begin{equation}\label{eq:fixenv-tau-repeated}
 \frac{c_{s,\tau}u_n^{-(s-2)/\tau}}{(\log u_n)^{2(s-2)}}
 \le G_{s,n}^{\rm r}-t_n^s
 \le\frac{C_{s,\tau}u_n^{-(s-2)/\tau}}
              {(\log u_n)^{3(s-2)/2}}.
\end{equation}
For $0\le d\le1$, use repeated vectors on $\mathcal A_d$ from
\eqref{eq:fixenv-selected-levels}, and long vectors elsewhere.
Denote the resulting equal-weighted Moran measure by $\mu_{\tau,d}$.
For $0<d\le1$, summation along
$n_j(d)\asymp_d(j(\log j)^2)^{1/d}$ shows that the sum of these
non-mean terms converges exactly when
$(s-2)/\tau\ge d$. At equality the upper bound is at most a constant
times $1/(j(\log j)^2)$; below equality the lower bound has a power
of $j$ strictly less than one and hence diverges. For $d=0$ the
non-mean terms are summable for all $s>2$.
Thus, for $s>2$, the full criterion is
\begin{equation}\label{eq:fixenv-tau-exact}
 \sum_nG_{s,n}(\mu_{\tau,d})<\infty
 \quad\Longleftrightarrow\quad
 s>\tau\ \text{ and }\ s\ge2+\tau d,
 \qquad 0\le d\le1.
\end{equation}
For $d=0$ the second condition is automatic.
It follows that
$
 s_c(\mu_{\tau,d})=\max\{2,\tau,2+\tau d\}.
$
Every $s_0\in[a_\tau,\tau+2]$ occurs: take
$d=(s_0-2)/\tau\in[0,1]$. For $p_0=s_0/(s_0-1)$, set
$\mu_{\tau,p_0}:=\mu_{\tau,d}$. This proves the claimed interval after
duality. If $s_0>a_\tau$, both conditions in \eqref{eq:fixenv-tau-exact}
hold at $s=s_0$, so the endpoint is attained. If $s_0=a_\tau>2$,
then $s_0=\tau$ and the common term $\sum_nt_n^{\tau}$ diverges.
If $s_0=a_\tau=2$, singularity excludes restriction at $p=2$.

For completeness, all the additional fixed data survive this change of
parameters. Each digit set is still a complete system of representatives
modulo $M_n$, so the common exponential spectrum and its completeness proof
are unchanged. Also $\log Q_N\sim N(\log N)^2$, so
\eqref{eq:fixenv-cylinder-mass} and \eqref{eq:fixenv-ball-masses} still imply singular
continuity, $D_q=1$, and local dimension one everywhere. The exponential
sums in \eqref{eq:fixenv-refined-envelope} now give
$|\widehat\mu_{\tau,d}(\xi)|\le C_\tau u_N^{-1/\tau}$ for
$Q_{N-1}\le|\xi|<Q_N$, proving \eqref{eq:fixenv-tau-envelope}.
The common resonances follow again from
$\widehat\mu_{\tau,d}(Q_n/2)=t_n\widehat\nu_n(1/2)$ and
\eqref{eq:fixenv-tail-uniform}. The same resonance-interval argument proves
$\dim_{\mathrm F}^\theta\mu_{\tau,d}=\theta$, because $t_n$ is still a negative
power of $n$, while $Q_n$ grows exponentially in $n(\log n)^2$.
Finally, Proposition~\ref{prop:fixed-envelope-obstruction} shows that no critical exponent
outside $I_\tau$ is possible under the stated hypotheses.
\end{proof}

At every interior point of $I_\tau$, endpoint attainment can vary
while the fixed data remain unchanged.
\begin{corollary}\label{cor:fixed-endpoint-separation}
For every $0<\tau<\infty$ and $p_0\in\operatorname{int}I_\tau$, there are two
singular continuous spectral Moran measures with the common scales,
digit cardinalities, exponential spectrum, maximum sequence and sharp Fourier envelope
of Theorem~\ref{thm:fixed-envelope-general-tau}, whose exact restriction
ranges are $[1,p_0]$ and $[1,p_0)$, respectively. Both satisfy
\eqref{eq:fixenv-dimensions} and have local dimension one at every support
point. For $\tau=2$, these are the fixed data of
Theorem~\ref{thm:fixed-envelope}, and $4/3<p_0<2$.
\end{corollary}
\begin{proof}
Put $s_0=p_0/(p_0-1)$ and $d=(s_0-2)/\tau$. Then
$\max\{2,\tau\}<s_0<\tau+2$ and $0<d<1$.
The measure $\mu_{\tau,d}$ constructed in
Theorem~\ref{thm:fixed-envelope-general-tau} attains the endpoint.
For the second measure, retain the same two digit choices at every level,
put $b=2\tau+2$, and use the repeated vectors at
\[
 \widetilde n_j=
 \left\lceil\left(\frac{j}{(\log(j+e))^b}\right)^{1/d}\right\rceil,
 \qquad j\ge j_0(d,\tau),
\]
and the long vectors elsewhere. These indices are strictly increasing
for sufficiently large $j_0$, because $d<1$.

At $s=s_0$, the lower bound in \eqref{eq:fixenv-tau-repeated} gives
\[
 G_{s_0,\widetilde n_j}^{\rm r}-t_{\widetilde n_j}^{s_0}
 \gtrsim_{s_0,\tau}
 \frac{(\log j)^{b-2(s_0-2)}}j,
\]
whose sum diverges, since $b-2(s_0-2)>2$. For $s>s_0$, put
$a=(s-2)/\tau>d$. The upper
bound in the same estimate is
\[
 G_{s,\widetilde n_j}^{\rm r}-t_{\widetilde n_j}^{s}
 \lesssim_{s,\tau,d}
 j^{-a/d}(\log j)^{ba/d-3(s-2)/2},
\]
which is summable. The common mean terms are summable for $s\ge s_0$
because $s_0>\tau$, and the remaining long-vector contributions are
summable for every $s>2$. Monotonicity therefore shows that the full
moment series converges exactly for $s>s_0$.
Theorem~\ref{thm:exact-moment-criterion} gives the open endpoint range.
All other assertions in Theorem~\ref{thm:fixed-envelope-general-tau}
hold for every selection of its two digit choices, so the fixed data
and dimension conclusions are preserved.
\end{proof}

\section{Finite-field spectral models and exact thresholds}\label{sec:restriction}

Combining Theorem~\ref{thm:exact-moment-criterion} with the
finite-field constructions below, we obtain spectral Moran
measures with any prescribed critical restriction exponent
in $[1,2]$. For the finite-field digit systems built from the tridiagonal
quadratic forms in Lemma~\ref{lem:q41-bases}, it reduces
to a scalar series for any scale sequence; logarithmic and secondary
logarithmic growth then determine the critical exponent and its endpoint.

\subsection{Arithmetic digit systems and common dimensions}

We use the arithmetic digit sets of Lai, Shi and Xie \cite{LSX}. Their one-level Weil estimate will also be used for the weighted dimension applications in Section~\ref{sec:weil}. Let $r,q$ be integers with $1\le r<q$, put $d=1+\lceil(q-r)/r\rceil$, and let $p>d$ be an odd prime. Let $\F_p$ be the finite field with $p$ elements and let $\F_{p^r}$ be its degree-$r$ extension; see, for example, \cite{LidlNiederreiter}. Since $\F_{p^r}$ is an $r$-dimensional vector space over $\F_p$, fix an ordered basis of $\F_{p^r}$ over $\F_p$. List the coordinate functions of $x,x^2,\ldots$ in this basis and denote their first $q$ entries by $g_0(x),\ldots,g_{q-1}(x)$, represented in $\{0,\ldots,p-1\}$. Define
\begin{equation}\label{eq:weil-digits}
B(p,r,q)=\left\{\sum_{j=0}^{q-1}g_j(x)p^j:x\in\F_{p^r}\right\}.
\end{equation}
We retain the uniform mask notation $m_B$ from Section~\ref{sec:fixed-envelope}.
The first $r$ digits make this map injective, so $|B(p,r,q)|=p^r$.
The argument of Lai, Shi and Xie \cite[Theorem~3.1]{LSX} yields
\begin{equation}\label{eq:weil-bound}
\|m_{B(p,r,q)}-D_{p^q}\|_\infty
\le(d-1)p^{-r/2}\bigl(2\log(2p)\bigr)^q.
\end{equation}
Although stated there for a power basis, the estimate remains valid
for any ordered basis: the proof uses only the bijectivity of the
coordinate map and the linear independence of its coordinate functionals,
together with the Weil bound for polynomial additive character sums.

For every fixed odd prime $P$ and integers $q_n\ge2$, put
$r_n=q_n-1$, $R_n=P^{q_n}$, $M_n=P^{q_n-1}$, and
$B_n:=B(P,r_n,q_n)$, and use the equal-weighted laws on $B_n$.
To verify spectrality, take
$L_n=P\{-(M_n-1)/2,\ldots,(M_n-1)/2\}$ and
$\Lambda_N=L_1+R_1L_2+\cdots+Q_{N-1}L_N$.
The residues of $B(P,r_n,q_n)$ modulo $M_n=P^{r_n}$ run through all
classes once, so $(R_n,B_n,L_n)$ is a Hadamard triple, as in
\cite[Proposition~6.1]{LSX}. The frequency bound
$|\lambda|\le(Q_N-1)/2$ for $\lambda\in\Lambda_N$ follows as in
\cite[Lemma~6.3]{LSX}: write
$\lambda=\sum_{n=1}^NQ_{n-1}\ell_n$ with $\ell_n\in L_n$ and sum
$Q_{n-1}|\ell_n|\le(Q_n-Q_{n-1})/2$.
Since $C_n=P$, Lemma~\ref{lem:bounded-frame-tail} gives
$\inf_{N,\lambda\in\Lambda_N}|\widehat{\nu_N}(\lambda/Q_N)|>0$.
The completeness criterion \cite[Theorem~1.3]{LaiWang}, also recalled in
\cite[Theorem~6.2]{LSX}, shows that $\bigcup_N\Lambda_N$ is an exponential spectrum. No mask approximation estimate is needed for this argument.

\begin{lemma}\label{lem:fixed-prime-common}
Let $q_n\ge2$ be integers such that
$
 q_n\rightarrow\infty$ and $
 \frac{q_n}{\sum_{j\le n}q_j}\rightarrow0.
$
For any choice of ordered bases, let $\mu_{\mathbf q}$ be the
equal-weighted Moran measure with digit sets $B(P,q_n-1,q_n)$.
Then $\mu_{\mathbf q}$ is singular continuous and spectral, and
\[
 \dimF\mu_{\mathbf q}=0,\qquad
 D_q(\mu_{\mathbf q})=1\ (1<q\le\infty),\qquad
 \dim_{\mathrm F}^{\theta}\mu_{\mathbf q}=\theta\ (0\le\theta\le1).
\]
Moreover, uniformly for $x\in\supp\mu_{\mathbf q}$,
\[
 \lim_{r\downarrow0}\frac{\log\mu_{\mathbf q}(B(x,r))}{\log r}=1.
\]
In particular, the measure and its support have Hausdorff dimension one.
\end{lemma}
\begin{proof}
Write $T_N=(\log P)\sum_{j\le N}q_j$. Every nonempty level-$N$
cylinder has mass $P^N/Q_N$, and their total length is $P^{-N}$.
Thus the measure is singular and continuous. The preceding Hadamard
argument proves spectrality.

Choose $N$ with $Q_N^{-1}\le r<Q_{N-1}^{-1}$. A cylinder containing
$x\in\supp\mu_{\mathbf q}$ lies in $B(x,r)$, up to zero-mass endpoints,
while the ball meets at most $2rQ_N+3$ grid intervals. Hence
\begin{equation}\label{eq:fixed-prime-common-mass}
 \frac{P^N}{Q_N}\le\mu_{\mathbf q}(B(x,r))\le5rP^N.
\end{equation}
The upper bound holds for every center in $\R$. The hypotheses imply
$N/T_{N-1}\to0$ and $T_N/T_{N-1}\to1$. Setting $L=-\log r$ gives
\[
 1-\frac{N\log P+\log5}{L}
 \le\frac{-\log\mu_{\mathbf q}(B(x,r))}{L}
 \le\frac{T_N-N\log P}{T_{N-1}}\longrightarrow1.
\]
This proves the local-dimension assertion. The upper mass bound also
gives $D_\infty=1$, hence $D_q=1$ for all $q>1$ and Hausdorff
dimension one for the measure and its support.

Discrete Parseval gives
$\sum_{k=1}^{R_n-1}|m_n(k/R_n)|^2=P-1$. Hence, by conjugate
symmetry, some $1\le k_n\le R_n/2$ satisfies
$|m_n(k_n/R_n)|\ge((P-1)/(R_n-1))^{1/2}$.
Lemma~\ref{lem:bounded-frame-tail} then gives, with
$\xi_n=k_nQ_{n-1}$,
$
 |\widehat\mu_{\mathbf q}(\xi_n)|\gtrsim_P R_n^{-1/2}.
$
Since $\log R_n=o(\log Q_n)$, we have
$\log\xi_n\sim\log Q_n$ and no positive power decay is possible.
Thus $\dimF\mu_{\mathbf q}=0$. The mass bounds imply finite
$t$-energy for every $t<1$, while singularity excludes Sobolev dimension
greater than one. Therefore $\dimS\mu_{\mathbf q}=1$, and
\cite[Corollary~1.6]{FraserSpectrum} gives the stated Fourier spectrum.
\end{proof}

\subsection{One-level moment estimates and arbitrary scale sequences}

For a single level let $r=q-1$ and $V=\F_{P^r}$. If
$\ell:V\to\F_P$ is the first coordinate functional, then
$\ell(x)=\operatorname{Tr}(\beta x)$ for some $\beta\ne0$;
here $\operatorname{Tr}$ is the trace from $V$ to $\F_P$;
see \cite[Lemma~2.2]{LSX}.
The last digit is the quadratic form $\mathcal Q(x)=\ell(x^2)$,
with associated nondegenerate symmetric bilinear form
$
 B_\beta(x,y)=\operatorname{Tr}(\beta xy).
$

Lemma~\ref{lem:q41-bases} below constructs an ordered $\F_P$-basis of $V$.
With $\ell$ taken as its first coordinate functional, the only mixed terms
of $\mathcal Q(x)=\ell(x^2)$ are $x_jx_{j+1}$, all with nonzero coefficients.
The resulting grid estimate will also provide a uniform mask bound for
every odd prime $P$.

\begin{lemma}
\label{lem:q41-bases}
For every $r\ge2$, there exists a basis $e_0,\ldots,e_{r-1}$ of $V$
such that the first coordinate of
$(\sum_jx_je_j)^2$ has the form
\begin{equation}\label{eq:q41-path-form}
 \mathcal Q(x_0,\ldots,x_{r-1})
 =\sum_{j=0}^{r-1}d_jx_j^2
     +2\sum_{j=0}^{r-2}b_jx_jx_{j+1},
 \qquad b_j\in\F_P\setminus\{0\}.
\end{equation}
Such bases are admissible in the digit construction \eqref{eq:weil-digits}.
\end{lemma}

\begin{proof}
Choose $\beta\in V$ with $\operatorname{Tr}(\beta)=1$, which is possible because
the trace map is onto. Put
$W=\{x:B_\beta(1,x)=0\}$. Since $B_\beta(1,1)=1$, the restriction
of $B_\beta$ to $W$ is nondegenerate, and
$V=\F_P1\mathbin\perp W$.

For \eqref{eq:q41-path-form}, start with an orthogonal basis
$u_1,\ldots,u_{r-1}$ of $W$, with
$a_j=B_\beta(u_j,u_j)\ne0$. Define
$
 e_j=u_j+u_{j+1}\quad(1\le j<r-1)$ and $ e_{r-1}=u_{r-1}.$
These vectors form a basis of $W$. Their Gram matrix is tridiagonal
and $B_\beta(e_j,e_{j+1})=a_{j+1}\ne0$.
Choose $v\in W$ with
$
 B_\beta(v,e_1)=1,~~
 B_\beta(v,e_j)=0~(2\le j\le r-1),
$
and set $e_0=1+v$. Nondegeneracy on $W$ guarantees the existence
of $v$. We have $\ell(e_0)=1$ and $\ell(e_j)=0$ for $j\ge1$,
where $\ell(x)=B_\beta(1,x)$. Thus $\ell$ is exactly the first
coordinate functional in this basis. Moreover
$B_\beta(e_0,e_1)=1$, and $B_\beta(e_0,e_j)=0$ for $j\ge2$.
This proves \eqref{eq:q41-path-form}.
\end{proof}

\begin{proposition}
\label{prop:q41-path}
Use a basis from \eqref{eq:q41-path-form}, or $e_0=1$ when $r=1$,
and let $B\subset\{0,\ldots,R-1\}$ be its digit set, where
$R=P^{r+1}$ and $|B|=P^r$. For its uniform
mask $m$, write $G_s=\sum_{k=1}^{R-1}|m(k/R)|^s$. Then
$
 \max_{1\le k<R}|m(k/R)|\le P R^{-1/2}.
$
Consequently, for every fixed $s>2$,
\begin{equation*}
 \sum_{k=1}^{R-1}|m(k/R)|^s\asymp_{P,s}R^{1-s/2}.
\end{equation*}

Moreover,
\begin{equation}\label{eq:q41-path-continuous}
 \|m-D_R\|_\infty\lesssim P R^{-1/2}\log(2R).
\end{equation}
\end{proposition}

\begin{proof}
Put $M=P^r$, and represent each coordinate by an integer in
$\{0,\ldots,P-1\}$. The digit attached to $x$ is
$
 b(x)=\sum_{j=0}^{r-1}x_jP^j+M[\mathcal Q(x)]_P.
$
Here $[a]_P$ denotes the representative of $a\in\F_P$ in
$\{0,\ldots,P-1\}$.
If $P\mid k$ and $k\not\equiv0\pmod R$, the last coordinate
contributes one and the remaining sum is the nonzero Fourier
coefficient of the uniform law on $\{0,\ldots,M-1\}$. Thus
$m(k/R)=0$.

Suppose $c=k\bmod P\ne0$. Apart from normalization by $P^r$, the
mask is
\begin{equation}\label{eq:q41-path-sum}
 \sum_{x\in\F_P^r}
 {e}_P(-c\mathcal Q(x))\prod_{j=0}^{r-1}z_j^{x_j},
 \qquad z_j={e}^{-2\pi i kP^j/R},
 \quad{e}_P(t)={e}^{2\pi i t/P}.
\end{equation}
For adjacent coordinates let
$T_j=({e}_P(-2cb_jab))_{a,b\in\F_P}$. Since $2cb_j\ne0$, we have $
 T_jT_j^*=P I $ and $ \|T_j\|_{2\to2}=\sqrt P.$
Define the diagonal unitary matrices
\[
 A_j=\operatorname{diag}\bigl(e_P(-cd_j a^2)z_j^a\bigr)_{a\in\F_P},
 \qquad 0\le j<r.
\]
For $r\ge2$, the sum in \eqref{eq:q41-path-sum} is exactly
$
 \mathbf1^{T}A_0T_0A_1T_1\cdots T_{r-2}A_{r-1}\mathbf1,
$
where $\mathbf1$ is the all-ones vector in $\mathbb C^P$. Its modulus is
at most $\|\mathbf1\|_2^2\prod_{j=0}^{r-2}\|T_j\|=P^{(r+1)/2}$.
After normalization this gives $P^{(1-r)/2}=P R^{-1/2}$.
The case $r=1$, with its basis $e_0=1$, also satisfies this bound.

Finally, discrete Parseval gives $G_2=P-1$. Hence
$
 G_s\le(P-1)(P R^{-1/2})^{s-2},
$
and H\"older gives the matching lower estimate
$G_s\ge(P-1)^{s/2}(R-1)^{1-s/2}$.

Apply the Fourier interpolation estimate used in
\eqref{eq:fixenv-continuous-mask} to $m-D_R$. Its value at zero is
zero, and all other grid values are bounded by $P R^{-1/2}$.
The bound $\sup_t\sum_{k=0}^{R-1}|D_R(t-k/R)|\lesssim\log(2R)$
proves \eqref{eq:q41-path-continuous}.
\end{proof}

\begin{theorem}\label{thm:tridiagonal-arbitrary}
Fix any odd prime $P$ and any sequence of integers $q_n\ge2$.
At each level use the digit system in Proposition~\ref{prop:q41-path},
with $r_n=q_n-1$ and $R_n=P^{q_n}$. The resulting measure
$\mu_{\mathbf q}^{\rm tr}$ is singular continuous and spectral. For every
$s>2$,
\begin{equation}\label{eq:tridiagonal-arbitrary-criterion}
 \begin{split}
 L^2(\mu_{\mathbf q}^{\rm tr})\to L^s(\R)\text{ extension holds}
 &\quad\Longleftrightarrow\quad
 \widehat\mu_{\mathbf q}^{\rm tr}\in L^s(\R)\\
 &\quad\Longleftrightarrow\quad
 \sum_{n=1}^{\infty}P^{-q_n(s/2-1)}<\infty.
 \end{split}
\end{equation}
No growth or monotonicity condition on $(q_n)$ is required.
If in addition $q_n\to\infty$ and $q_n/\sum_{j\le n}q_j\to0$,
all the dimension conclusions of Lemma~\ref{lem:fixed-prime-common} hold.
\end{theorem}
\begin{proof}
The one-level frame bounds are $C_n=P$. Proposition~\ref{prop:q41-path}
gives $G_{s,n}\asymp_{P,s}P^{-q_n(s/2-1)}$, so
Theorem~\ref{thm:exact-moment-criterion} proves
\eqref{eq:tridiagonal-arbitrary-criterion}.
The Hadamard and tail argument above gives spectrality for every scale
sequence. The level-$N$ cylinder cover has total length $P^{-N}$, and
each cylinder has mass at most $P^{-N}$ because $M_n\ge P$.
Thus the measure is singular and continuous. The last assertion is
Lemma~\ref{lem:fixed-prime-common}.
\end{proof}

The exact series criterion separates the effect of leading logarithmic
growth from the secondary term that decides the endpoint.
\begin{corollary}\label{thm:tridiagonal-logarithmic}
Fix an odd prime $P$, $A>0$ and $B\in\R$, and choose integers $q_n\ge2$
such that, as $n\to\infty$,
$
 q_n=A\log n+B\log\log n+O(1).
$
For the tridiagonal measure of Theorem~\ref{thm:tridiagonal-arbitrary},
\begin{equation*}
 p_{\rm res}(\mu_{\mathbf q}^{\rm tr})
 =p_*(A,P):=\frac{2A\log P+2}{A\log P+2}.
\end{equation*}
Its exact restriction range is $[1,p_*(A,P)]$ if $B>A$, and
$[1,p_*(A,P))$ if $B\le A$. In particular,
$q_n=\lceil A\log n\rceil+2$ gives endpoint failure.
\end{corollary}
\begin{proof}
For each fixed $s>2$,
\[
 G_{s,n}\asymp
 n^{-A(s/2-1)\log P}(\log n)^{-B(s/2-1)\log P}.
\]
The critical Fourier-moment exponent in
\eqref{eq:critical-moment-exponent} is
$s_*=2+2/(A\log P)$. At this exponent the series is comparable to
$\sum_{n\ge2}[n(\log n)^{B/A}]^{-1}$, which converges precisely when $B>A$.
Apply Theorem~\ref{thm:tridiagonal-arbitrary} and duality;
restriction at $p=1$ is trivial, while singularity excludes $p=2$.
\end{proof}

The logarithmic family and the two extreme summability regimes give the
full range of critical restriction exponents with common dimension data.
\begin{theorem}\label{thm:complete-nondeterminacy}
Fix any odd prime $P$. For every $p_0\in[1,2]$, there is a singular
continuous spectral Moran measure $\mu_{p_0}$ such that
\[
 D_q(\mu_{p_0})=1\quad(1<q\le\infty),\qquad
 \dim_{\mathrm F}^{\theta}\mu_{p_0}=\theta\quad(0\le\theta\le1),
\]
and
\[
 \lim_{r\downarrow0}
 \frac{\log\mu_{p_0}(B(x,r))}{\log r}=1
 \qquad(x\in\supp\mu_{p_0}),
 \qquad p_{\rm res}(\mu_{p_0})=p_0.
\]
For every $1<p_0<2$, one may choose either exact restriction range
$[1,p_0]$ or $[1,p_0)$. All these measures have the same pointwise
multifractal spectrum: the local-dimension level set is empty for
$\alpha\ne1$, and is the whole support, of Hausdorff dimension one,
for $\alpha=1$.
\end{theorem}
\begin{proof}
For $1<p_0<2$, set
$A=2(p_0-1)/((2-p_0)\log P)$ and apply
Corollary~\ref{thm:tridiagonal-logarithmic}; choose $B=2A$ for the
closed endpoint or $B=0$ for the open endpoint. For $p_0=2$, take
$q_n=n+2$, for which the series in
\eqref{eq:tridiagonal-arbitrary-criterion} converges for every $s>2$.
For $p_0=1$, take $q_n=\lceil\sqrt{\log(n+1)}\rceil+2$; then its
terms are $n^{-o(1)}$ for each fixed $s>2$, and the series diverges.
In every case $q_n\to\infty$ and $q_n/\sum_{j\le n}q_j\to0$.
Theorem~\ref{thm:tridiagonal-arbitrary} and
Lemma~\ref{lem:fixed-prime-common} prove the remaining assertions.
\end{proof}

\subsection{Quantitative decay in the superlogarithmic regime}

Fix an odd prime $P$ so large that
$
\rho_P:=\frac{2\log(2P)}{\sqrt P}<1.
$
For $\kappa>0$, let
$
q_n=\lceil n^\kappa\rceil+2,
~ r_n=q_n-1,
$
and use the equal-weighted finite-field digit set
$B_n=B(P,r_n,q_n)$.  Thus
$
R_n=P^{q_n}$ and $ M_n=P^{q_n-1}=R_n/P.$

Since $r_n=q_n-1$, the degree parameter in \eqref{eq:weil-bound}
is $d=2$ at every level, so that estimate becomes
\begin{equation}\label{eq:fixed-prime-error}
\maskerr_n
\le P^{1/2}\rho_P^{q_n}
\le C\exp(-c n^\kappa).
\end{equation}
In particular the mask errors are summable.  Notice that no rapid domination of $Q_{n-1}$ by $R_n$ occurs: indeed
$
\log Q_n\asymp n^{\kappa+1},
~
\frac{\log R_n}{\log Q_n}\asymp\frac1n\longrightarrow0.
$

\begin{theorem}\label{thm:zeroF-restriction}
Let $\kappa>0$ and let $\mu_\kappa$ be the equal-weighted Moran measure defined above.  Then $\mu_\kappa$ is a singular continuous spectral probability measure and
\begin{equation*}
\dimF\mu_\kappa=0,
\quad
\dimtwo\mu_\kappa=\dimH\mu_\kappa=\dimH\supp\mu_\kappa=1.
\end{equation*}
Moreover there are constants $c,C>0$ such that
\begin{equation}\label{eq:stretched-Fourier}
|\widehat\mu_\kappa(\xi)|
\le C\exp\!\left[-c(\log(2+|\xi|))^{\kappa/(\kappa+1)}\right]
\qquad(\xi\in\mathbb R),
\end{equation}
and, for every sufficiently small $r>0$,
\begin{equation}\label{eq:stretched-Frostman}
\sup_{x\in\mathbb R}\mu_\kappa(B(x,r))
\le Cr\exp\!\left[C(\log(e/r))^{1/(\kappa+1)}\right].
\end{equation}
The decay exponent in \eqref{eq:stretched-Fourier} is sharp at the level of stretched logarithms: there are $\xi_n\to\infty$ and $C'>0$ such that
\begin{equation}\label{eq:stretched-lower}
|\widehat\mu_\kappa(\xi_n)|
\ge \exp\!\left[-C'(\log(2+|\xi_n|))^{\kappa/(\kappa+1)}\right].
\end{equation}
For every $1\le p<2$,
\begin{equation*}
\|\widehat f\|_{L^2(\mu_\kappa)}
\le C_p\|f\|_{L^p(\mathbb R)},
\qquad f\in\mathcal S(\mathbb R).
\end{equation*}
Consequently $p_{\rm res}(\mu_\kappa)=2$, while restriction fails at the endpoint $p=2$.
\end{theorem}

\begin{proof}
The sequence $q_n=\lceil n^\kappa\rceil+2$ satisfies the assumptions of
Lemma~\ref{lem:fixed-prime-common}, which proves all the qualitative
dimension, spectrality and singular-continuity assertions.

For the quantitative decay, apply the telescoping argument leading to
\eqref{eq:fixenv-refined-envelope} and drop its first minimum.
By \eqref{eq:fixed-prime-error}, if
$Q_N\le |\xi|<Q_{N+1}$, then
\[
|\widehat\mu_\kappa(\xi)|
\lesssim |\xi|^{-1}
+\sum_{n\le N}\maskerr_n\frac{Q_n}{Q_N}
+\sum_{n>N}\maskerr_n.
\]
The $n=N$ term is $O(e^{-cN^\kappa})$.  For $n<N$ one has
$Q_n/Q_N\le P^{-q_N}$, while $\sum_n\maskerr_n<\infty$; the tail is also
$O(e^{-c'N^\kappa})$.  Since
$
\log Q_N=(\log P)\sum_{n\le N}q_n\asymp N^{\kappa+1},
$
this proves \eqref{eq:stretched-Fourier}.  
For the lower estimate, the proof of Lemma~\ref{lem:fixed-prime-common}
provides $\xi_n=k_nQ_{n-1}$ with $1\le k_n\le R_n/2$ and
$
 |\widehat\mu_\kappa(\xi_n)|\gtrsim_P R_n^{-1/2}
 =P^{-q_n/2}.
$
Since $q_n\asymp n^\kappa$ and
$\log\xi_n\asymp\log Q_n\asymp n^{\kappa+1}$, this proves
\eqref{eq:stretched-lower}.

For the mass estimate, choose $N$ with
$Q_N^{-1}\le r<Q_{N-1}^{-1}$. Equation~\eqref{eq:fixed-prime-common-mass}
gives $\mu_\kappa(B(x,r))\le5rP^N$, and
$N\asymp(\log(e/r))^{1/(\kappa+1)}$ proves
\eqref{eq:stretched-Frostman}.

For restriction, $C_n=P$ and discrete Parseval gives $G_{2,n}=P-1$.
With $\eta_n$ as in \eqref{eq:one-level-quantities}, for every $s>2$ we have
\[
 G_{s,n}\le(P-1)\eta_n^{s-2}
 \lesssim_{P,s}\exp[-c_s n^\kappa].
\]
Theorem~\ref{thm:exact-moment-criterion} therefore gives restriction
for every $1<p<2$; $p=1$ is trivial. Singularity excludes a bounded
density, so Proposition~\ref{prop:quadratic-endpoint} excludes $p=2$.
\end{proof}

\begin{remark}\label{rem:restriction-phase}
Applied to \eqref{eq:stretched-Fourier} and \eqref{eq:stretched-Frostman},
the gauge criterion of Polasek--Rela \cite[Theorem~1.2]{PR} compares the
Fourier decay power $\kappa/(\kappa+1)$ with the mass-loss power
$1/(\kappa+1)$. For $\kappa>1$ it gives restriction for every $p<2$.
At $\kappa=1$ the constants in the two bounds give a range of $p>1$
sufficiently close to one. For $0<\kappa<1$, these bounds do not give
a nontrivial range through that criterion.
Theorem~\ref{thm:zeroF-restriction}, however, gives restriction for every
$p<2$ for all $\kappa>0$, by retaining the one-level Fourier moments.
\end{remark}

\subsection{Resonant lower bounds and arbitrary bases}

The same sampled resonances give a quantitative obstruction even when the
one-level frame bounds are not uniformly bounded.
\begin{proposition}\label{prop:Ls-resonant-obstruction}
Suppose $\varepsilon_n\to0$, and put
$C_{2,n}=R_n\sum_bw_{n,b}^2$. For every $2\le s<\infty$, there is
$n_0$ such that, with extended nonnegative values allowed,
\begin{align*}
 \|\widehat\mu\|_{L^s(\R)}^s
 &\gtrsim_s\sum_{n\ge n_0}G_{s,n}\\
 &\ge\sum_{n\ge n_0}(R_n-1)^{1-s/2}(C_{2,n}-1)^{s/2}.
\end{align*}
Here $G_{2,n}=C_{2,n}-1$. In particular, if $\liminf_nC_{2,n}>1$,
divergence of $\sum_nR_n^{1-s/2}$ rules out restriction at $p=s'$ and
at every larger $p\le2$.
\end{proposition}
\begin{proof}
The sampling estimate \eqref{eq:q41-sampling} is valid for every
$s\ge2$. Lemma~\ref{lem:bounded-frame-tail} gives a positive lower
bound for $|\widehat\nu_n(u)|$ on $|u|\le1/2$ for all sufficiently
large $n$. At $kQ_{n-1}$, $1\le k\le R_n/2$, all preceding masks are
one. Conjugate symmetry and the disjointness of these integer frequency
sets therefore give
\[
 \sum_{n\ge n_0}G_{s,n}
 \lesssim_s\sum_{j\in\Z}|\widehat\mu(j)|^s
 \lesssim_s\|\widehat\mu\|_s^s.
\]
If the last norm is infinite the inequality is automatic; otherwise it
follows directly from the sampling estimate. Discrete Parseval gives
$G_{2,n}=C_{2,n}-1$, and H\"older on the $R_n-1$ nonzero residues gives
$
 G_{s,n}\ge(R_n-1)^{1-s/2}(C_{2,n}-1)^{s/2}.
$
The stated consequence follows by testing extension with $g=1$ and
interpolating any larger restriction exponent with the trivial $p=1$
estimate.
\end{proof}

For arbitrary ordered bases, fix an odd prime $P$ with $\rho_P<1$ and
integers $q_n\ge2$ tending to infinity. For the equal-weighted measure
$\mu_{\mathbf q}$ on $B(P,q_n-1,q_n)$, we have $C_n=C_{2,n}=P$ and
\begin{equation}\label{eq:q-sequence-coherence}
 \eta_n\le\varepsilon_n\le P^{1/2}\rho_P^{q_n}.
\end{equation}
The next bounds retain this freedom of basis. The tridiagonal case has
already been characterized exactly in
Theorem~\ref{thm:tridiagonal-arbitrary} for every odd prime.

\begin{theorem}\label{thm:scale-growth-transition}
For the arbitrary-basis family just defined, the following hold.
\begin{enumerate}
\item[\textup{(i)}] If $q_n/\log(n+1)\to\infty$, restriction holds
exactly for $1\le p<2$.
\item[\textup{(ii)}] If $q_n=\lceil A\log n\rceil+2$ with $A>0$, then
\begin{equation*}
 \frac{2A\log(1/\rho_P)+1}{A\log(1/\rho_P)+1}
 \le p_{\rm res}(\mu_{\mathbf q})
 \le\frac{2A\log P+2}{A\log P+2}.
\end{equation*}
Restriction holds strictly below the left-hand exponent and fails at
and above the right-hand exponent.
\item[\textup{(iii)}] If $q_n/\log(n+1)\to0$, restriction holds
exactly at $p=1$.
\end{enumerate}
\end{theorem}
\begin{proof}
For every $s>2$, Parseval and \eqref{eq:q-sequence-coherence} give
\begin{equation*}
 (P-1)^{s/2}(R_n-1)^{1-s/2}
 \le G_{s,n}
 \le(P-1)\eta_n^{s-2}
 \lesssim_{P,s}\rho_P^{(s-2)q_n}.
\end{equation*}
Apply Theorem~\ref{thm:exact-moment-criterion}. In case \textup{(i)},
the upper bound is summable for every $s>2$; singularity excludes
$p=2$. In case \textup{(ii)}, it is summable when
$A(s-2)\log(1/\rho_P)>1$, giving the left-hand exponent. The lower
bound is comparable to $n^{-A(s/2-1)\log P}$ and has divergent sum
for $2<s\le2+2/(A\log P)$, including equality. This gives the
right-hand exponent and its failure. In case \textup{(iii)}, the
lower-bound terms are $n^{-o(1)}$ for each fixed $s>2$, so every such
series diverges. Interpolation with $p=1$ and singularity give the
remaining endpoint assertions.
\end{proof}

\section{Fourier decay, resonances and correlation dimension}\label{sec:general-proof}

This section derives Fourier and correlation dimension formulas.
They complement the exact moment criterion in
Theorem~\ref{thm:exact-moment-criterion}.
Section~\ref{sec:weil} uses these formulas to compare the classical
Fourier-decay and Frostman criteria \cite{Mockenhaupt,Mitsis,BakSeeger}
with the Fourier-spectrum and $L^q$-dimension criteria
\cite{CFdORestriction,CFdOLq}.
Put
\begin{equation}\label{onelevelq}
 v_n=\sum_{j=0}^{R_n-1}(w_{n,j}-R_n^{-1})^2,
\end{equation}
which measures the quadratic nonuniformity of the level-$n$ digit law.

\begin{theorem}\label{thm:general}
Put $S_n=\sum_{j=0}^{R_n-1}w_{n,j}^2=R_n^{-1}+v_n$. Suppose that
$\maskerr_n\to0$ and that, for every $\eta>0$, there is $C_\eta<\infty$ such
that
\begin{equation}\label{eq:flatness-assumption}
 \maskerr_n^2\le C_\eta Q_n^\eta v_n\quad(n\ge1).
\end{equation}
Then
\begin{align}
\dimF\mu&=\min\{1,\gamma_{\mathrm F}\},
&\gamma_{\mathrm F}&=\liminf_{n\to\infty}\frac{-\log v_n}{\log Q_n},
\label{eq:general-fourier}\\
\dimtwo\mu&=\gamma_2,
&\gamma_2&=\liminf_{n\to\infty}
\frac{-\sum_{j=1}^n\log S_j}{\log Q_n}.
\label{eq:general-correlation}
\end{align}
\end{theorem}

The formulas require no regularity assumption on the growth of the scales
$R_n$, such as bounds on $R_{n+1}/R_n$ or on
$\log R_n/\log Q_{n-1}$.  Parseval and Cauchy--Schwarz give
$v_n\le\maskerr_n^2\le R_nv_n$, so \eqref{eq:flatness-assumption} is
automatic when $\log R_n=o(\log Q_n)$, provided $\maskerr_n\to0$.
The finite-field families in Section~\ref{sec:weil} also satisfy this
comparison under the parameter conditions stated there, including regimes
with strongly varying scales.

We now turn to the proof of Theorem~\ref{thm:general}.

\subsection{Uniform-tail approximations, Fourier decay and resonances}

Recall the independent digit variables from Section~\ref{sec:criterion}.
Under the assumptions of Theorem~\ref{thm:general}, cylinder endpoints
have zero $\mu$-mass. Indeed, Fourier inversion gives
$|w_{n,j}-R_n^{-1}|\le\maskerr_n$. Since $R_n\ge2$ and
$\maskerr_n\to0$, we have
$\max_j w_{n,j}\le R_n^{-1}+\maskerr_n\le3/4$ for all sufficiently
large $n$. Thus every fixed digit sequence has probability zero.
Each endpoint has at most two digit expansions, so it also has zero
$\mu$-mass. Consequently, the level-$N$ interval associated with a fixed
prefix $J_1,\ldots,J_N$ has length $Q_N^{-1}$ and mass
$\prod_{j=1}^Nw_{j,J_j}$.

Write $\mathcal P_N$ for the partition of $[0,1)$ into the intervals
$
\left[\frac{k}{Q_N},\frac{k+1}{Q_N}\right),
~~ 0\le k<Q_N.
$ Put
\begin{equation*}
P_N=\prod_{j=1}^NS_j=\sum_{I\in\mathcal P_N}\mu(I)^2,
\qquad P_0=1.
\end{equation*}
Let
\[
\mu_N
=
\left(
\bigconv_{n=1}^{N}
\sum_{j=0}^{R_n-1}
w_{n,j}\delta_{j/Q_n}
\right)
\conv
\left(Q_N\mathbf 1_{[0,Q_N^{-1}]}\,dx\right).
\] That is,  \(\mu_N\) agrees with \(\mu\) through level \(N\), and the part of the construction below scale \(Q_N^{-1}\) is replaced by the uniform probability measure on \([0,Q_N^{-1}]\). Then $\mu_N\Rightarrow\mu$, since each tail replacement moves a point by at most $Q_N^{-1}$. Moreover,
\begin{equation}\label{eq:finite-product}
\widehat{\mu_N}(\xi)=\prod_{j=1}^Nm_j(\xi/Q_j)\widehat\unif(\xi/Q_N),
\qquad
\widehat\mu(\xi)=\prod_{j=1}^{\infty}m_j(\xi/Q_j),
\end{equation}
where $\unif$ is Lebesgue probability measure on $[0,1]$.

The approximating measures $\mu_N$ also admit a Riesz product representation.
Define the $1$-periodic function $g_n$ by
\[
 g_n(x)=R_nw_{n,j},\qquad
 x\in[j/R_n,(j+1)/R_n),\quad 0\le j<R_n.
\]
Then
\begin{equation}\label{eq:riesz-density}
	\dd\mu_N(x)=f_N(x)\dd x,\qquad
	f_N(x)=\prod_{n=1}^N g_n(Q_{n-1}x),
	\qquad x\in[0,1).
\end{equation}

We now use these approximations to estimate Fourier decay.
A comparison with Dirichlet masks is also used in
\cite[Lemma~5.3 and Proposition~5.4]{LSX}.
Here we allow arbitrary scale growth and nonuniform digit weights.

\begin{lemma}\label{lem:telescoping}
If $0<\beta<1$ and $\maskerr_n\le A Q_n^{-\beta}$ for all $n$, then
$
|\widehat\mu(\xi)|\le C_\beta(1+A)(1+|\xi|)^{-\beta}.
$
\end{lemma}
\begin{proof}
The identity
$
\widehat\unif(\xi/Q_{n-1})
=D_{R_n}(\xi/Q_n)\widehat\unif(\xi/Q_n)
$
and \eqref{eq:finite-product} give, with $\mu_0=\unif$,
\begin{equation}\label{eq:telescoping}
|\widehat{\mu_n}(\xi)-\widehat{\mu_{n-1}}(\xi)|
\le\maskerr_n\min\{1,Q_n/|\xi|\}.
\end{equation}
For $|\xi|\ge1$, choose $N\ge0$ with $Q_N\le|\xi|<Q_{N+1}$,
using $Q_0=1$. The first sum below is zero when $N=0$. Since
$Q_{n+1}\ge2Q_n$, the two geometric sums satisfy
\[
\frac1{|\xi|}\sum_{1\le n\le N}Q_n^{1-\beta}\lesssim_\beta|\xi|^{-\beta},
\qquad
\sum_{n>N}Q_n^{-\beta}\lesssim_\beta|\xi|^{-\beta}.
\]
Sum \eqref{eq:telescoping}, pass to the weak limit, and use
$|\widehat\unif(\xi)|\lesssim|\xi|^{-1}$ to obtain the bound for
$|\xi|\ge1$. The range $|\xi|<1$ follows from $|\widehat\mu(\xi)|\le1$.
\end{proof}

The following proposition relates polynomial Fourier decay to the
quadratic nonuniformity $v_n$ (see \eqref{onelevelq}) of the individual levels.
\begin{proposition}\label{prop:resonance}
Suppose $\maskerr_n\to0$. If $0<t<1$ and
$|\widehat\mu(\xi)|\lesssim(1+|\xi|)^{-t/2}$, then
$
v_n\le C Q_n^{-t}.
$
Here $C$ is independent of $n$ and may depend on $t$, the Fourier
decay constant, and finitely many initial levels of the fixed construction.
\end{proposition}
\begin{proof}
By Lemma~\ref{lem:bounded-frame-tail}, for all sufficiently large $n$
and $1\le k\le R_n/2$,
\[
 |m_n(k/R_n)|
 \lesssim |\widehat\mu(kQ_{n-1})|
 \lesssim(kQ_{n-1})^{-t/2}.
\]
The first estimate uses the identity
$\widehat\mu(kQ_{n-1})=m_n(k/R_n)\widehat\nu_n(k/R_n)$;
all preceding masks are one. Discrete Parseval and conjugate symmetry give
\[
 R_nv_n=\sum_{k=1}^{R_n-1}|m_n(k/R_n)|^2
 \le2\sum_{k=1}^{\lfloor R_n/2\rfloor}|m_n(k/R_n)|^2
 \lesssim Q_{n-1}^{-t}R_n^{1-t},
\]
since $0<t<1$. Divide by $R_n$ and absorb the finitely many initial
levels into the implicit constant.
\end{proof}

\begin{proof}[Proof of Theorem~\ref{thm:general}: formula \eqref{eq:general-fourier}]
If $0<t<\min\{1,\gamma_{\mathrm F}\}$, choose $\eta>0$ with $t+\eta<\gamma_{\mathrm F}$. Then $v_n\lesssim Q_n^{-t-\eta}$ and \eqref{eq:flatness-assumption} gives $\maskerr_n\lesssim Q_n^{-t/2}$. Lemma~\ref{lem:telescoping} implies $\dimF\mu\ge t$. This proves the lower bound, including the case $\gamma_{\mathrm F}=+\infty$; if $\gamma_{\mathrm F}=0$ the bound is immediate.

Conversely, for every $0<t<\dimF\mu$, Proposition~\ref{prop:resonance} gives $-\log v_n\ge t\log Q_n-O(1)$, where the bounded term is independent of $n$; the inequality is automatic when $v_n=0$. Therefore $\gamma_{\mathrm F}\ge t$. Letting $t\uparrow\dimF\mu$ proves the upper bound.
\end{proof}

The argument determines the critical decay exponent, but does not in general imply that the corresponding endpoint decay estimate is attained.

\subsection{Correlation estimates between consecutive scales}

\begin{lemma}\label{lem:fejer}
Let $J,J'$ be independent digits with law $(w_j)_{0\le j<R}$. Put
$
m(t)=\sum_{j=0}^{R-1}w_j e^{-2\pi ijt},
$
and $\maskerr=\|m-D_R\|_\infty$. For $R^{-1}\le b\le1$,
\begin{equation*}
\Prob\left(\dist\left(\frac{J-J'}R,\Z\right)\le b\right)
\lesssim b+\maskerr^2.
\end{equation*}
The implicit constant is absolute.
\end{lemma}
\begin{proof}
The bound is immediate for $b\ge1/4$. Otherwise put
$K=\lfloor(2b)^{-1}\rfloor$, so $2\le K<R$ and $K\asymp b^{-1}$.
The Fej\'er kernel
\[
 F_K(x)=\frac1K\left|\sum_{l=0}^{K-1}e^{2\pi ilx}\right|^2
 =\sum_{|k|<K}(1-|k|/K)e^{2\pi ikx}
\]
satisfies $K^{-1}F_K(x)\ge4/\pi^2$ when $\dist(x,\Z)\le b$:
use periodicity and $K\dist(x,\Z)\le1/2$ in the sine-quotient
formula. Independence therefore yields
\begin{align*}
 \Prob\bigl(\dist((J-J')/R,\Z)\le b\bigr)
 &\le\frac{\pi^2}{4K}\,
       \E F_K((J-J')/R)\\
 &=\frac{\pi^2}{4K}
   \sum_{|k|<K}(1-|k|/K)|m(k/R)|^2\\
 &\lesssim K^{-1}+\varepsilon^2\lesssim b+\varepsilon^2.
\end{align*}
Here $m(0)=1$, and $|m(k/R)|\le\varepsilon$ for $0<|k|<K<R$
because the uniform mask vanishes at those grid points.
\end{proof}

\begin{lemma}\label{lem:correlation-window}
If $Q_n^{-1}\le r<Q_{n-1}^{-1}$ and $u=rQ_{n-1}$, then
\begin{equation*}
cP_{n-1}(u+S_n)\le\Corr_\mu(r)
\le CP_{n-1}(u+\maskerr_n^2),
\end{equation*}
where $c,C>0$ are absolute constants.
\end{lemma}
\begin{proof}
For independent $X,X'$ with law $\mu$, write
\[
 X=(K+Y)/Q_{n-1},\quad X'=(K'+Y')/Q_{n-1},\qquad
 Y=(J_n+Z_n)/R_n,
\]
where the integer prefixes $K,K'$ are independent of the corresponding
tails, and $Y'$ is an independent copy of $Y$. If $p_k=\Prob(K=k)$,
then $\sum_kp_k^2=P_{n-1}$. Also $R_n^{-1}\le u<1$ and
\[
 \Corr_\mu(r)=\Prob(|K-K'+Y-Y'|\le u).
\]
Since $Y,Y'\in[0,1]$, this event requires $K-K'\in\{-1,0,1\}$.
For each such prefix difference it also requires
$
 \dist((J_n-J_n')/R_n,\Z)
 \le u+R_n^{-1}\le2u.
$
Lemma~\ref{lem:fejer}, or the trivial bound when $2u\ge1$, bounds
this tail probability by $C(u+\varepsilon_n^2)$. For each prefix
difference $h$, Cauchy--Schwarz gives
$\sum_kp_kp_{k-h}\le P_{n-1}$. Summing over $h=-1,0,1$ proves
the upper bound.

For the lower bound, restrict to $K=K'$, an event of probability
$P_{n-1}$. Conditional on it, $J_n=J_n'$ implies $|Y-Y'|\le u$,
giving probability at least $S_n$. Partitioning $[0,1]$ into at most
$2/u$ intervals of length at most $u$ gives a second lower bound
$\Prob(|Y-Y'|\le u)\ge u/2$, by Cauchy--Schwarz. Hence this probability
is at least $(S_n+u)/4$, which proves the result.
\end{proof}

\begin{proof}[Completion of the proof of Theorem~\ref{thm:general}:  formula \eqref{eq:general-correlation}]
Since $R_n^{-1}\le S_n\le1$, we have $0\le\gamma_2\le1$.
The bound $\Corr_\mu(Q_n^{-1})\ge P_n$ gives $\dimtwo\mu\le\gamma_2$.
The reverse inequality is immediate if $\gamma_2=0$. Otherwise fix $0<t<t'<\gamma_2$ and choose $0<\eta<t'-t$. Then $P_n\lesssim Q_n^{-t'}$ for every $n$. Since $v_n\le S_n$, condition \eqref{eq:flatness-assumption} gives
\[
P_{n-1}\maskerr_n^2\lesssim Q_n^\eta P_{n-1}S_n
=Q_n^\eta P_n\lesssim Q_n^{-t'+\eta}.
\]
For $r,u,n$ as in Lemma~\ref{lem:correlation-window}, the first term in its upper bound satisfies
$
P_{n-1}u\lesssim Q_{n-1}^{-t'}u
\le Q_{n-1}^{-t}u^t=r^t,
$
and the second is at most $CQ_n^{-t}\le Cr^t$. Thus $\Corr_\mu(r)\lesssim r^t$ for all sufficiently small $r$. Letting $t\uparrow\gamma_2$ proves the result and completes Theorem~\ref{thm:general}.
\end{proof}

\subsection{Quantitative Fourier decay without summability}
\label{sec:profiles}

The cancellation of each mask error at the origin gives a quantitative
bound even when the mask errors are not summable. The next proposition
extends \eqref{eq:fixenv-refined-envelope} to arbitrary weighted Moran
measures.
\begin{proposition}\label{prop:profile-envelope}
Every weighted Moran measure satisfies, for $|\xi|\ge1$,
\begin{equation}\label{eq:general-refined-envelope}
 |\widehat\mu(\xi)|\le\frac C{|\xi|}
 +C\sum_{n\ge1}\varepsilon_n
 \min\left\{1,\frac{|\xi|}{Q_{n-1}}\right\}
 \min\left\{1,\frac{Q_n}{|\xi|}\right\}.
\end{equation}
The series converges for each fixed $\xi$, without a summability
assumption on $(\varepsilon_n)$. In particular, if $\varepsilon_n\to0$,
then $\widehat\mu(\xi)\to0$ as $|\xi|\to\infty$.
\end{proposition}
\begin{proof}
Put $\mathcal E_n=m_n-D_{R_n}$. Then $\mathcal E_n(0)=0$, and Bernstein's inequality
gives $\|\mathcal E_n'\|_\infty\le2\pi(R_n-1)\varepsilon_n$.
For the uniform-tail approximations in \eqref{eq:finite-product},
with $\mu_0=\unif$, the identity
$\widehat\unif(\xi/Q_{n-1})
=D_{R_n}(\xi/Q_n)\widehat\unif(\xi/Q_n)$ yields
\[
 |\widehat\mu_n(\xi)-\widehat\mu_{n-1}(\xi)|
 \lesssim\varepsilon_n
 \min\{1,|\xi|/Q_{n-1}\}\min\{1,Q_n/|\xi|\}.
\]
For fixed $\xi$, the tail of this series is dominated by
$2|\xi|\sum_n Q_{n-1}^{-1}<\infty$, since $\varepsilon_n\le2$.
Telescoping and weak convergence prove \eqref{eq:general-refined-envelope}.

If $Q_{N-1}\le|\xi|<Q_N$, the series is bounded by a constant times
\[
 \sum_{n<N}\varepsilon_n2^{-(N-1-n)}
 +\varepsilon_N
 +\sum_{n>N}\varepsilon_n2^{-(n-N-1)}.
\]
When $\varepsilon_n\to0$, the last sum and the central term tend to
zero. In the first sum, finitely many early terms tend to zero and the
remaining terms have arbitrarily small total after a sufficiently late
starting index. This proves the final assertion.
\end{proof}

\section{Weighted models: dimensions and restriction comparisons}\label{sec:weil}

We now specialize the general dimension formulas to weighted Moran measures determined by digit sets \(B_n\subset\{0,\ldots,R_n-1\}\) and level-dependent weights, and derive explicit formulas for their Fourier, correlation, and \(L^q\) dimensions.
\subsection{Weighted masks and dimension formulas}

For the principal weighted specialization, let
$B_n\subset\{0,\ldots,R_n-1\}$, $M_n=|B_n|$, and assume
\begin{equation}\label{eq:weil-abstract}
R_n\to\infty,~~ M_n\le R_n/3,~~
\frac{\log M_n}{\log R_n}\to s\in(0,1],~~
\|m_{B_n}-D_{R_n}\|_\infty\le M_n^{-1/2}R_n^{o(1)},
\end{equation}
Here $s$ denotes the limiting digit dimension, and $m_{B_n}$ is the
uniform mask of $B_n$. For
$0\le\alpha_n\le1$, put
\begin{equation}\label{eq:weighted-law}
 w_{n,j}=\frac{1-\alpha_n}{R_n}
 +\frac{\alpha_n}{M_n}\one_{B_n}(j).
\end{equation}
These weights are nonnegative and satisfy $\sum_jw_{n,j}=1$ for every
$0\le\alpha_n\le1$.

The following theorem applies Theorem~\ref{thm:general} to the mixed digit laws \eqref{eq:weighted-law}.
\begin{theorem}\label{thm:weighted}
Under \eqref{eq:weil-abstract} and \eqref{eq:weighted-law}, set
\[
\ell_n=\log R_n,\quad T_n=\log Q_n,\quad
\delta_n=\frac{\ell_n}{T_n},\quad
a_n=\frac{\log M_n-2\log\alpha_n}{\ell_n},
\]
where $a_n=+\infty$ if $\alpha_n=0$. Then
\begin{equation}\label{eq:variable-main}
\dimF\mu=\min\{1,\liminf_n a_n\delta_n\},\qquad
\dimtwo\mu=\liminf_n\frac{\sum_{j=1}^n\min\{1,a_j\}\ell_j}{T_n}.
\end{equation}
If $\alpha_n=R_n^{-\tau/2}$ for a fixed $\tau\ge0$ and
$\delta=\liminf_n\delta_n$, then
\begin{equation}\label{eq:constant-main}
\dimF\mu=\min\{1,(s+\tau)\delta\},\qquad
\dimtwo\mu=\min\{1,s+\tau\}.
\end{equation}
\end{theorem}

The arithmetic family from Section~\ref{sec:restriction}, based on
\cite[Theorem~3.1]{LSX}, realizes \eqref{eq:weil-abstract} for every
$s\in(0,1]$. Choose integers $q_n\to\infty$, with
$r_n=\lfloor sq_n\rfloor$ when $s<1$ and $r_n=q_n-1$ when $s=1$,
taking $q_n$ sufficiently large that $2\le r_n<q_n$. Choose odd primes
$p_n>d_n$ tending to infinity, where
$d_n=1+\lceil(q_n-r_n)/r_n\rceil$, and put
\[
R_n=p_n^{q_n},\quad M_n=p_n^{r_n},\quad B_n=B(p_n,r_n,q_n).
\]
Then $M_n/R_n\le1/p_n\le1/3$ and
$\log M_n/\log R_n=r_n/q_n\to s$. Thus $s=1$ is compatible with
$M_n\le R_n/3$. Since $d_n$ is bounded,
\[
\frac{\log\left[(d_n-1)(2\log(2p_n))^{q_n}\right]}{\log R_n}
=\frac{\log(2\log(2p_n))}{\log p_n}+o(1)\longrightarrow0.
\]
Together with \eqref{eq:weil-bound}, this proves \eqref{eq:weil-abstract}.
Neither the rapid-growth condition $R_n>Q_{n-1}^n$ nor the prefactor bound
$(d_n-1)(2\log(2p_n))^{q_n}\le p_n^{1/2}$ from
\cite[(4.4)--(4.5)]{LSX} is needed here.

For the weights \eqref{eq:weighted-law}, direct calculation gives
\begin{align}
m_n-D_{R_n}&=\alpha_n(m_{B_n}-D_{R_n}),\label{eq:weighted-mask-identity}\\
S_n&=R_n^{-1}+\alpha_n^2(M_n^{-1}-R_n^{-1}),\notag\\
v_n&=\alpha_n^2(M_n^{-1}-R_n^{-1}).\notag
\end{align}
These identities allow us to apply Theorem~\ref{thm:general}; the resonance
estimate in Proposition~\ref{prop:resonance} determines the dependence on
scale growth.

\begin{proof}[Proof of Theorem~\ref{thm:weighted}]
Since $M_n\le R_n/3$, we have
\begin{equation}\label{eq:v-comparison}
\frac23\frac{\alpha_n^2}{M_n}\le v_n\le\frac{\alpha_n^2}{M_n}.
\end{equation}
By \eqref{eq:weil-abstract} and \eqref{eq:weighted-mask-identity},
$\maskerr_n^2\le R_n^{o(1)}\alpha_n^2/M_n\lesssim R_n^{o(1)}v_n$; when $\alpha_n=0$, both sides vanish. Also $\maskerr_n\le R_n^{-s/2+o(1)}\to0$. Thus Theorem~\ref{thm:general} applies.

For $\alpha_n>0$, \eqref{eq:v-comparison} gives $-\log v_n=a_n\ell_n+O(1)$, uniformly in the weights. Moreover,
\begin{equation}\label{eq:S-comparison}
S_n\asymp R_n^{-1}+R_n^{-a_n},\qquad
-\log S_n=\min\{1,a_n\}\ell_n+O(1),
\end{equation}
with the same interpretation for $a_n=+\infty$. Since $R_n\to\infty$, we have $n/T_n\to0$, so the accumulated $O(1)$ errors in \eqref{eq:S-comparison} do not affect the limit. Equations \eqref{eq:general-fourier}--\eqref{eq:general-correlation} yield \eqref{eq:variable-main}. If $\alpha_n=R_n^{-\tau/2}$, then $a_n\to s+\tau$; a weighted Ces\`aro argument proves \eqref{eq:constant-main}.
\end{proof}

For completeness, the special Riesz factors in this case are
\[
\psi_n(x)=\frac{R_n}{M_n}\one_{E_n}(x)-1,
\quad E_n=\bigcup_{b\in B_n}[b/R_n,(b+1)/R_n),
\]
extended $1$-periodically. Formula \eqref{eq:riesz-density} becomes
$f_N(x)=\prod_{n=1}^N(1+\alpha_n\psi_n(Q_{n-1}x))$.

\subsection{Dimension equality, positive Fourier decay, and scale growth}

Formulas~\eqref{eq:variable-main} and~\eqref{eq:constant-main} give a direct criterion for
dimension equality and positive Fourier dimension in terms of the
relative growth of consecutive scales.  We state the fixed-power consequence first, before considering variable amplitudes.
\begin{corollary}\label{cor:equality}
Let $a=s+\tau$ in the power-law family of Theorem~\ref{thm:weighted}. Then
\[
\dimF\mu=\dimtwo\mu
\quad\Longleftrightarrow\quad
\begin{cases}
\delta=1,&0<a\le1,\\
\delta\ge a^{-1},&a>1.
\end{cases}
\]
Positive Fourier dimension holds if and only if $\delta>0$, equivalently if there is $A>0$ such that $R_n\ge Q_{n-1}^A$ eventually. If $\log R_n/\log Q_{n-1}\to A\in(0,\infty)$, then
\begin{equation*}
\dimF\mu=\min\left\{1,\frac{Aa}{A+1}\right\}.
\end{equation*}
\end{corollary}
\begin{proof}
Apply \eqref{eq:constant-main} and the identity
$
\delta_n=\frac{1}{1+\log Q_{n-1}/\log R_n}.
$
The equivalence for positive dimension follows by bounding this ratio away from zero. All statements are unaffected by finitely many initial levels.
\end{proof}

In particular, for $a>1$, the sufficient condition $R_n\ge Q_{n-1}^{1/(a-1)}$ preserves full Fourier dimension. For $0<a\le1$, equality is impossible when
$\log R_n/\log Q_{n-1}\to A<\infty$. If $\log R_{n+1}/\log R_n\to b>1$, then $\delta=(b-1)/b$ and the Fourier dimension is $\min\{1,a(b-1)/b\}$. Ordinary lacunarity of the frequencies $Q_n$ is much weaker:
$Q_{n+1}/Q_n=R_{n+1}\ge2$ holds automatically in the present weighted
Moran construction.

\begin{proposition}\label{prop:equality-rigidity}
Under the hypotheses of Theorem~\ref{thm:weighted}, let
$(\alpha_n)$ be an arbitrary sequence with $0\le \alpha_n\le1$, and  put
$f=\dimF\mu$, $d=\dimtwo\mu$, and $\Delta=\limsup_n\delta_n$. If $d<1$, then
\begin{equation}\label{eq:gap-rigidity}
f\le\Delta d.
\end{equation}
Consequently, $0<\dimF\mu=\dimtwo\mu<1$ requires $\limsup_n\delta_n=1$.
\end{proposition}
\begin{proof}
The assertion is immediate when $f=0$. Suppose $f>0$. If $\Delta=0$, then for any $0<t<f$, \eqref{eq:variable-main} forces $a_n\delta_n\ge t$ eventually and hence $a_n\to\infty$. This would give $d=1$, a contradiction. Thus $\Delta>0$.

For any $0<t<f$ and $\eta>0$, eventually $a_n\delta_n\ge t$ and $\delta_n\le\Delta+\eta$. Hence $\min\{1,a_n\}\ge\min\{1,t/(\Delta+\eta)\}$. Taking the weighted lower limit in \eqref{eq:variable-main} gives
$d\ge\min\{1,t/(\Delta+\eta)\}$. Since $d<1$, let $t\uparrow f$ and $\eta\downarrow0$ to obtain \eqref{eq:gap-rigidity}.
\end{proof}

\subsection{\texorpdfstring{$L^q$-dimensions at intermediate scales}
	{Lq dimensions at intermediate scales}}\label{sec:Lq}

When successive scales are widely separated, intermediate scales may
contribute beyond the cylinder moments. We first account for these scales
and then specialize to the weighted family. Here $q$ denotes a moment
order, independently of the finite-field integers $q_n$.

Throughout this subsection, let $\mu$ be a non-atomic weighted Moran measure.
We use the dimension and grid notation from Section~\ref{sec:dimension-notation}
and write $T_n=\log Q_n$. For $1<q<\infty$, put
\begin{equation*}
S_n(q)=\sum_{j=0}^{R_n-1}w_{n,j}^q,\quad
P_N(q)=\prod_{n=1}^NS_n(q)
=\sum_{I\in\mathcal P_N}\mu(I)^q,
\quad A_N(q)=-\log P_N(q),
\end{equation*}
with $P_0(q)=1$ and $A_0(q)=0$. In the notation of
Section~\ref{sec:general-proof}, $S_n(2)=S_n$ and $P_N(2)=P_N$.
For $1\le k\le R_n$, partition $\{0,\ldots,R_n-1\}$ into consecutive blocks
\[
A_{n,k,l}=\{lk,\ldots,\min((l+1)k,R_n)-1\},
\quad 0\le l<\lceil R_n/k\rceil,
\]
and define their moment
\begin{equation*}
H_{n,q}(k)=\sum_l\left(\sum_{j\in A_{n,k,l}}w_{n,j}\right)^q.
\end{equation*}

The following lemma and theorem show that the $L^q$-dimension is determined
by digit-block moments at all intermediate scales and give sufficient
conditions for the cylinder-scale formula to hold.

\begin{lemma}\label{lem:block-comparison}
Fix $1<q<\infty$. For $1\le k\le R_n$ and $r=k/Q_n$,
\begin{equation}\label{eq:block-comparison}
 \mathcal M_q(\mu,r)\asymp_q P_{n-1}(q)H_{n,q}(k),
\end{equation}
where the implicit constants depend only on $q$, and are independent of
$n$ and $k$.
\end{lemma}
\begin{proof}
In each level-$(n-1)$ interval, group its level-$n$ subintervals into
consecutive blocks of $k$. A grouped interval has mass
$p\sum_{j\in A_{n,k,l}}w_{n,j}$, where $p$ is its parent mass, so the sum
of the $q$th powers of the grouped masses is exactly
$P_{n-1}(q)H_{n,q}(k)$. Except for the last interval in each parent, every
grouped interval has length $r$, and the last has length at most $r$.
Each member of the fixed grid $\mathcal D_r$ meets at most four grouped
intervals, while each grouped interval meets at most two members of
$\mathcal D_r$. Hence
\[
 \sum_{i=1}^m a_i^q\le\left(\sum_{i=1}^m a_i\right)^q
 \le m^{q-1}\sum_{i=1}^m a_i^q
\]
with $m\le4$ gives \eqref{eq:block-comparison}, with constants depending
only on $q$.
\end{proof}

\begin{theorem}\label{thm:all-scale-Lq}
For every $1<q<\infty$,
\begin{equation}\label{eq:all-scale-Lq}
D_q(\mu)=\liminf_{n\to\infty}\ \min_{1\le k\le R_n}
\frac{A_{n-1}(q)-\log H_{n,q}(k)}
{(q-1)(T_n-\log k)}.
\end{equation}
In particular,
\begin{equation}\label{eq:cylinder-upper-Lq}
D_q(\mu)\le d_q^{\rm cyl}:=
\liminf_{N\to\infty}\frac{A_N(q)}{(q-1)T_N}.
\end{equation}
Equality holds if $\log R_n=o(T_n)$. Alternatively,  equality holds whenever, for every $\eta>0$,
\begin{equation}\label{eq:block-regularity}
H_{n,q}(k)\le C_{q,\eta}
\bigl((k/R_n)^{q-1}+Q_n^\eta S_n(q)\bigr)
\qquad(1\le k\le R_n).
\end{equation}
\end{theorem}

\begin{proof}
Let $n\ge2$ and $r\in[Q_n^{-1},Q_{n-1}^{-1}]$. Choose $1\le k\le R_n$ so that
$
 r\le \frac{k}{Q_n}<2r.
$
For two fixed grids whose mesh sizes differ by at most a factor of two,
each interval in either grid meets only a bounded number of intervals in
the other. The same elementary $q$-moment comparison, together with
Lemma~\ref{lem:block-comparison}, therefore gives
\[
 \mathcal M_q(\mu,r)\asymp_q
 \mathcal M_q(\mu,k/Q_n)
 \asymp_q P_{n-1}(q)H_{n,q}(k),
\]
with constants independent of $n,k,r$. Moreover,
$
 \log(1/r)=T_n-\log k+O(1),
$
and $T_n-\log k\ge T_{n-1}\to\infty$. Thus all bounded
multiplicative errors, as well as the additive $O(1)$ in the denominator,
disappear in the logarithmic quotient. Taking the lower limit over all
$r\downarrow0$ proves \eqref{eq:all-scale-Lq}; taking $k=1$ proves
\eqref{eq:cylinder-upper-Lq}. Moreover, $R_n^{1-q}\le S_n(q)\le1$
implies $0\le d_q^{\rm cyl}\le1$.

If $T_n/T_{n-1}\to1$, the estimate
$\mathcal M_q(\mu,r)\lesssim_q P_{n-1}(q)$ on this window proves the reverse bound in \eqref{eq:cylinder-upper-Lq}: each grid interval of length $r\le Q_{n-1}^{-1}$ meets at most two level-$(n-1)$ intervals, and subdivision cannot increase their total $q$-moment. The hypothesis $\log R_n=o(T_n)$ implies $T_n/T_{n-1}\to1$. For the other assertion, equality is immediate if $d_q^{\rm cyl}=0$; otherwise use \eqref{eq:block-regularity}. Fix $0<t<t'<d_q^{\rm cyl}$ and $0<\eta<(q-1)(t'-t)$. Then
$P_j(q)\lesssim Q_j^{-(q-1)t'}$. Writing $u=k/R_n$ and $r=u/Q_{n-1}$, \eqref{eq:block-comparison} gives
\[
\mathcal M_q(\mu,r)
\lesssim Q_{n-1}^{-(q-1)t'}u^{q-1}
+Q_n^{\eta-(q-1)t'}\lesssim r^{(q-1)t}.
\]
For the first term, $u\le1$ and $t'<1$ give
$Q_{n-1}^{-(q-1)t'}u^{q-1}\le r^{(q-1)t'}\le r^{(q-1)t}$.
For the second, $\eta-(q-1)t'<-(q-1)t$ and $r\ge Q_n^{-1}$ give the same bound. Let $t\uparrow d_q^{\rm cyl}$.
\end{proof}

The same argument gives the corresponding exact formula at \(q=\infty\), which will be used below. Put
\[
U_n=\max_jw_{n,j},\qquad
H_{n,\infty}(k)=\max_l\sum_{j\in A_{n,k,l}}w_{n,j}.
\]
The same partition comparison, now for the largest mass, gives
\begin{equation*}
D_\infty(\mu)=\liminf_{n\to\infty}\min_{1\le k\le R_n}
\frac{-\sum_{j<n}\log U_j-\log H_{n,\infty}(k)}{T_n-\log k}.
\end{equation*}
Thus both the sufficient moment data and the mass obstruction can be computed from finite digit blocks.

\subsection{Explicit moments of the weighted laws}

Return to \eqref{eq:weighted-law}. Write $s_n=\log M_n/\ell_n$ and
$h_n=-\log\alpha_n/\ell_n$, with $h_n=+\infty$ when $\alpha_n=0$. For $1<q<\infty$, the level moments satisfy
\begin{align}
S_n(q)&=(R_n-M_n)\left(\frac{1-\alpha_n}{R_n}\right)^q
+M_n\left(\frac{1-\alpha_n}{R_n}+\frac{\alpha_n}{M_n}\right)^q,
\label{eq:weighted-q-exact}\\
S_n(q)&\asymp_q R_n^{1-q}+\alpha_n^qM_n^{1-q}.
\label{eq:weighted-q-comparison}
\end{align}
The constants in \eqref{eq:weighted-q-comparison} are independent of $n$ and $\alpha_n$: the lower bounds follow from Jensen's inequality and the weights on $B_n$, and the upper bound follows from $(a+b)^q\le2^{q-1}(a^q+b^q)$.

\begin{theorem}\label{thm:weighted-Lq}
Under the hypotheses of Theorem~\ref{thm:weighted}, for every $1<q\le2$,
\begin{equation}\label{eq:variable-Lq}
D_q(\mu)=\liminf_{N\to\infty}\frac1{T_N}
\sum_{n=1}^N\min\left\{1,s_n+\frac{q}{q-1}h_n\right\}\ell_n.
\end{equation}
In particular, for $\alpha_n=R_n^{-h}$ with $h\ge0$,
\begin{equation}\label{eq:fixed-Lq}
D_q(\mu)=\min\left\{1,s+\frac{qh}{q-1}\right\},\quad 1<q\le2,
\qquad \dimF\mu=\min\{1,(s+2h)\delta\}.
\end{equation}
Here $h=\tau/2$ in the notation of Theorem~\ref{thm:weighted}. For each $q>2$, \eqref{eq:variable-Lq} remains valid under \eqref{eq:block-regularity} or under $\log R_n=o(T_n)$. Without either additional hypothesis, the exact formula is \eqref{eq:all-scale-Lq}, with
\begin{equation}\label{eq:weighted-block-exact}
H_{n,q}(k)=\sum_l\left(
\frac{(1-\alpha_n)|A_{n,k,l}|}{R_n}
+\frac{\alpha_n|B_n\cap A_{n,k,l}|}{M_n}\right)^q.
\end{equation}
\end{theorem}

\begin{proof}
Since $M_n=R_n^{s+o(1)}\to\infty$, we have
$\max_jw_{n,j}\le R_n^{-1}+M_n^{-1}\to0$, so $\mu$ is non-atomic.
Let $H^B_{n,q}(k)$ denote the block moment for the uniform law on $B_n$, and put $u=k/R_n$ and $e_n=\|m_{B_n}-D_{R_n}\|_\infty$. Two digits in the same block differ by at most $k$, so Lemma~\ref{lem:fejer} gives
$
H^B_{n,2}(k)\lesssim u+e_n^2.
$

For any probability vector $(b_l)$ and $1<q\le2$, Jensen's inequality with weights $b_l$ gives
$\sum_l b_l^q\le(\sum_l b_l^2)^{q-1}$. Thus
$
H^B_{n,q}(k)\lesssim_q u^{q-1}+e_n^{2(q-1)}.
$
The uniform digit law has block moment at most $u^{q-1}$. Applying convexity to the two parts of the mixture yields
\begin{align*}
H_{n,q}(k)
&\lesssim_q u^{q-1}+\alpha_n^qe_n^{2(q-1)}\\
&\le C_{q,\eta}\bigl(u^{q-1}+Q_n^\eta\alpha_n^qM_n^{1-q}\bigr)
\le C_{q,\eta}\bigl(u^{q-1}+Q_n^\eta S_n(q)\bigr).
\end{align*}
Indeed, $e_n\le M_n^{-1/2}R_n^{o(1)}$ by \eqref{eq:weil-abstract}, so for each fixed $q$ and $\eta>0$,
$e_n^{2(q-1)}\le C_{q,\eta}M_n^{1-q}Q_n^\eta$; the finitely many initial levels are absorbed into the constant.
This verifies precisely \eqref{eq:block-regularity}, without a loss on the uniform term. Theorem~\ref{thm:all-scale-Lq} now applies. By \eqref{eq:weighted-q-comparison},
\[
\frac{-\log S_n(q)}{q-1}
=\min\left\{1,s_n+\frac{q}{q-1}h_n\right\}\ell_n+O_q(1).
\]
Since $N/T_N\to0$, summing proves \eqref{eq:variable-Lq}. Weighted Ces\`aro convergence proves \eqref{eq:fixed-Lq}. The assertions for $q>2$ follow from Theorem~\ref{thm:all-scale-Lq}; no higher-moment flatness is inferred from the quadratic mask bound alone.
\end{proof}

For the original equal-weighted Salem measures in
\cite[Theorem~4.2]{LSX}, the special case $D_q(\mu)=s$ for $1<q\le2$
also follows from $s=\dimF\mu\le D_2(\mu)\le D_q(\mu)$ and the
cylinder upper bound \eqref{eq:cylinder-upper-Lq}.

For the weighted Moran measures defined by \eqref{eq:weighted-law} with
$\alpha_n=R_n^{-h}$, the $L^q$ dimensions for $q>2$ satisfy
the following bounds:
\begin{equation*}
\frac{q}{2(q-1)}\min\{1,s+2h\}
\le D_q(\mu)\le\min\left\{1,s+\frac{qh}{q-1}\right\},\qquad q>2.
\end{equation*}
The lower bound follows from $\mathcal M_q(\mu,r)\le\mathcal M_2(\mu,r)^{q/2}$; the upper bound is the cylinder bound. Similarly,
\begin{equation}\label{eq:Frostman-bracket}
\frac12\min\{1,s+2h\}\le D_\infty(\mu)
\le\min\{1,s+h\}.
\end{equation}
The largest cylinder mass is $\prod_{n\le N}\max_jw_{n,j}$, which proves the upper bound. The lower bound follows from
$\mu(B(x,r))^2\le\Corr_\mu(2r)$. The actual value is determined by the largest block masses through the
formula following Theorem~\ref{thm:all-scale-Lq}.

\begin{proposition}\label{prop:higher-moment-obstruction}
For each $q>2$ and $0<s<(q-2)/(q-1)$, there are digit sets $B_n$ satisfying \eqref{eq:weil-abstract},
equipped with the uniform weights
$
w_{n,j}=|B_n|^{-1}\mathbf 1_{B_n}(j),
$ with $\delta_n\to1$, such that $\dimF\mu=D_2(\mu)=s$ but $D_q(\mu)<s$. The modified digit sets in this assertion are not claimed to be Hadamard.
\end{proposition}
\begin{proof}
Choose $0<s<(q-2)/(q-1)$. Take $q_n=n+C_s$ and $r_n=\lfloor sq_n\rfloor$, with $C_s$ large enough that $r_n\ge2$, and choose the odd primes $p_n>d_n$ recursively so large that
$
q_n\log p_n\ge nT_{n-1}.
$
Then the digit sets $B_n=B(p_n,r_n,q_n)$ from \eqref{eq:weil-digits}
satisfy \eqref{eq:weil-abstract} by \cite[Theorem~3.1]{LSX}, and
\[
\delta_n=\frac{\log R_n}{T_{n-1}+\log R_n}\ge\frac{n}{n+1}\rightarrow1.
\]
Let $K_n=\lfloor\sqrt{M_n}\rfloor$. Modify each $B_n$ to contain all of $\{0,\ldots,K_n-1\}$, removing the same number of digits outside this interval. The resulting set $B_n'$ has cardinality $M_n$ and
$
\|m_{B_n'}-m_{B_n}\|_\infty\le2K_n/M_n\lesssim M_n^{-1/2}.
$
Let $\mu$ be the Moran measure with weights
$w_{n,j}=M_n^{-1}\one_{B_n'}(j)$. Thus \eqref{eq:weil-abstract}
is preserved for the modified sets, and Theorem~\ref{thm:weighted}
gives $\dimF\mu=D_2(\mu)=s$. On the other hand,
$H_{n,q}(K_n)\ge(K_n/M_n)^q$. Also
$A_{n-1}(q)\le(q-1)T_{n-1}=o(\ell_n)$ and
$\log K_n=(s/2+o(1))\ell_n$.
Substituting $k=K_n$ into \eqref{eq:all-scale-Lq} therefore gives
\[
D_q(\mu)\le\frac{qs/2}{(q-1)(1-s/2)}
=\frac{qs}{(q-1)(2-s)}<s.
\]
The cylinder exponent is still $s$. This proves that the higher-moment qualification in Theorem~\ref{thm:weighted-Lq} is necessary even within the abstract flat-mask family.
\end{proof}

\subsection{\texorpdfstring{$L^q$}{Lq}-dimension and Fourier-spectrum restriction bounds}

We now apply the restriction bounds of Carnovale, Fraser and de Orellana
to the dimension formulas above.

Let $\mu$ be a compactly supported probability measure on $\R$ and let
$1<q\le\infty$, with $q'=q/(q-1)$ and $\infty'=1$.
By \cite[Corollary~3.3]{CFdOLq}, if $\beta=\dimF\mu>0$, restriction
holds whenever
\begin{equation}\label{eq:Lq-restriction-threshold}
p'>2q'+\frac{4(1-D_q(\mu))}{\beta}.
\end{equation}
More generally, \cite[Theorem~3.2]{CFdOLq} gives restriction when
$0<\theta\le1$, $\sigma=\dim_{\mathrm F}^{\theta}\mu>\theta$, and
\begin{equation}\label{eq:joint-restriction-threshold}
p'>\frac{2q'(\sigma-\theta)+4(1-D_q(\mu))}
{\sigma-\theta D_q(\mu)}.
\end{equation}
For $\sigma=\infty$, the right-hand side is interpreted as its limit
$2q'$. At $q=\infty$, this bound also follows from
\cite[Theorem~3.1]{CFdORestriction}.

For $q=\infty$, \eqref{eq:Lq-restriction-threshold} is the usual
nonendpoint Stein--Tomas--Mockenhaupt bound.
We next recall the necessary conditions.

\begin{proposition}\label{prop:restriction-obstructions}
If restriction holds at $1<p\le2$, then
\begin{equation}\label{eq:mass-necessary}
\sup_x\mu(B(x,r))\lesssim r^{2/p'},\qquad 0<r<1.
\end{equation}
Also $\widehat\mu\in L^{p'}(\R)$; in particular, for $p<2$,
\begin{equation}\label{eq:spectrum-necessary}
\dim_{\mathrm F}^{2/p'}\mu\ge2/p'.
\end{equation}
The endpoint $p=2$ holds if and only if $\mu$ has a bounded density.
\end{proposition}

\begin{proof}
Choose $\psi\in\mathcal S(\R)$ with $|\widehat\psi|\ge1$ on $[-1,1]$ and apply restriction to
$f(t)=r e^{2\pi ixt}\psi(rt)$. Its Fourier transform is at least one in modulus on $B(x,r)$, and $\|f\|_p=r^{1/p'}\|\psi\|_p$, proving \eqref{eq:mass-necessary}. Taking $g=1$ in the extension estimate proves the Fourier assertion and hence \eqref{eq:spectrum-necessary}. For $p=2$, \eqref{eq:mass-necessary} implies a bounded density; the converse is Plancherel's theorem.
\end{proof}

The mass bound \eqref{eq:mass-necessary} is the one-dimensional mass-concentration case of the quantitative flatness obstruction in Fraser \cite[Theorem~2.3]{FraserFlatness}. In higher dimensions anisotropic concentration contains further information; only intervals are needed here.

For $\beta=\dimF\mu>0$,
define the sufficient extension threshold over a range of moment orders by
\[
\mathcal R_I(\mu):=\inf_{q\in I}
\left(2q'+\frac{4(1-D_q(\mu))}{\beta}\right),
\qquad I\in\{(1,\infty],(1,2]\}.
\]
Then $\mathcal R_{(1,\infty]}(\mu)\le\mathcal R_{(1,2]}(\mu)$.
Combining the sufficient estimate \eqref{eq:Lq-restriction-threshold}
with the necessary condition \eqref{eq:mass-necessary} gives
\begin{equation}\label{eq:restriction-sandwich}
\frac{\mathcal R_{(1,\infty]}(\mu)}{\mathcal R_{(1,\infty]}(\mu)-1}
\le p_{\rm res}(\mu)\le\frac{2}{2-D_\infty(\mu)}.
\end{equation}
One may further improve the lower bound by minimizing the thresholds in \eqref{eq:joint-restriction-threshold} over their stated domain. These are bounds for the critical exponent, not a general formula for it.

\subsection{An explicit restriction range beyond Stein--Tomas}

Return to the power-law family of Theorem~\ref{thm:weighted-Lq}, with
$0<s\le1$, $\alpha_n=R_n^{-h}$, $h\ge0$, and
$\delta=\liminf_n\delta_n>0$. Write
$
a=s+2h,~ \beta=\min\{1,a\delta\}=\dimF\mu.
$
Optimizing \eqref{eq:Lq-restriction-threshold} over the range $1<q\le2$ already gives an explicit answer. With $y=q'\ge2$, its threshold is
$
2y+\frac4\beta(1-s-hy)_+.
$
Consequently the best threshold supplied by these moment orders is
\begin{equation}\label{eq:closed-restriction-phase}
\mathcal R_{(1,2]}(\mu)=
\begin{cases}
4,&a\ge1,\\[2pt]
4+4(1-a)/\beta,&a<1,\ \beta\ge2h,\\[2pt]
2(1-s)/h,&a<1,\ 0<\beta<2h.
\end{cases}
\end{equation}
The last case requires $h>0$. Restriction holds for
$1\le p<\mathcal R_{(1,2]}(\mu)/(\mathcal R_{(1,2]}(\mu)-1)$.
If $a\ge1$, or if $a<1$ and $\beta\ge2h$, the minimum is attained at $q=2$. If $a<1$ and $\beta<2h$, it is attained at
\begin{equation*}
q_*=\frac{1-s}{1-s-h}\in(1,2),\qquad D_{q_*}(\mu)=1.
\end{equation*}
When $a<1$ and $\beta=2h$, every $q'$ between $2$ and $(1-s)/h$ gives the same threshold. These statements follow by minimizing a continuous piecewise affine function, and describe a transition in the available estimates.

The optimization leading to \eqref{eq:closed-restriction-phase} yields
a strictly larger restriction range than the classical Stein--Tomas
criterion in a nonempty parameter region.
\begin{corollary}\label{cor:beyond-ST}
If $0<h<(1-s)/2$ and $(s+2h)\delta<2h$, then the finite-moment restriction range is strictly larger than the range supplied by the ordinary Stein--Tomas criterion using the actual Fourier and Frostman dimensions of $\mu$.
\end{corollary}
\begin{proof}
By \eqref{eq:Frostman-bracket}, $D_\infty(\mu)\le s+h$. Therefore the classical extension threshold is at least
\[
2+\frac{4(1-s-h)}\beta
>\frac{2(1-s)}h=\mathcal R_{(1,2]}(\mu).
\]
The strict inequality is equivalent to $\beta<2h$; only the upper bound
$D_\infty(\mu)\le s+h$ is used.
\end{proof}

The following example quantifies this improvement over the classical
Stein--Tomas criterion. Main Theorems B and C show that such dimensional
data do not determine the optimal restriction range.
\begin{example}\label{ex:Lq-gain}
Choose the digit sets $B_n=B(p_n,r_n,q_n)$ from \eqref{eq:weil-digits}
with $s=1/4$ and $\delta_n\to1/3$.
Take $\alpha_n=R_n^{-1/4}$. Then
\begin{equation*}
\dimF\mu=\tfrac14,\quad
D_q(\mu)=\min\left\{1,\tfrac14+\tfrac{q}{4(q-1)}\right\}
\quad(1<q\le2).
\end{equation*}
Here $D_{3/2}=1$, $D_2=3/4$, and $D_\infty\le1/2$. The finite-moment argument proves restriction for every $p<6/5$, whereas the ordinary Stein--Tomas criterion cannot yield any $p>10/9$. The mass obstruction gives the rigorous interval
\begin{equation}\label{eq:Lq-example-bracket}
\frac65\le p_{\rm res}(\mu)\le\frac43.
\end{equation}
No equality or endpoint assertion is made in \eqref{eq:Lq-example-bracket}; the point is the strict quantitative improvement of the available restriction range.
\end{example}
\begin{proof}
We first specify the scales. Take $q_n=n+8$, $r_n=\lfloor q_n/4\rfloor$,
and an odd prime $p_1$ sufficiently large. Bertrand's postulate
\cite[Chapter~2]{AignerZiegler} states
that, for every integer $N\ge2$, there is a prime $p$ with $N<p<2N$.
For $n\ge2$, put $x_n=\exp(T_{n-1}/(2q_n))$ and apply this fact to
$N=\lceil x_n\rceil$ to choose a prime $p_n$ with
\[
 x_n<p_n<4x_n,\qquad
 \log p_n=\frac{T_{n-1}}{2q_n}+O(1).
\]
Then $\ell_n=T_{n-1}/2+O(q_n)$ and $T_n\ge(3/2)T_{n-1}$.
Since $q_n=O(n)$, these choices give $p_n\to\infty$,
$\ell_n/T_{n-1}\to1/2$, and hence $\delta_n\to1/3$.
The degree parameters $d_n=1+\lceil(q_n-r_n)/r_n\rceil$ are bounded,
so taking $p_1$ large also ensures $p_n>d_n$ at every level.
Since $r_n/q_n\to1/4$, the estimate \eqref{eq:weil-bound}
verifies \eqref{eq:weil-abstract} for these digit sets.

Theorem~\ref{thm:weighted-Lq}, with $s=h=1/4$, gives the stated Fourier
and $L^q$ dimensions. Equation~\eqref{eq:Frostman-bracket} gives
$D_\infty\le1/2$. Taking $q=3/2$ in
\eqref{eq:Lq-restriction-threshold} gives the sufficient threshold $p'>6$,
whereas the classical threshold is at least
$2+4(1-1/2)/(1/4)=10$.
Finally, \eqref{eq:mass-necessary} gives
$p_{\rm res}\le2/(2-D_\infty)\le4/3$, proving
\eqref{eq:Lq-example-bracket}.
\end{proof}

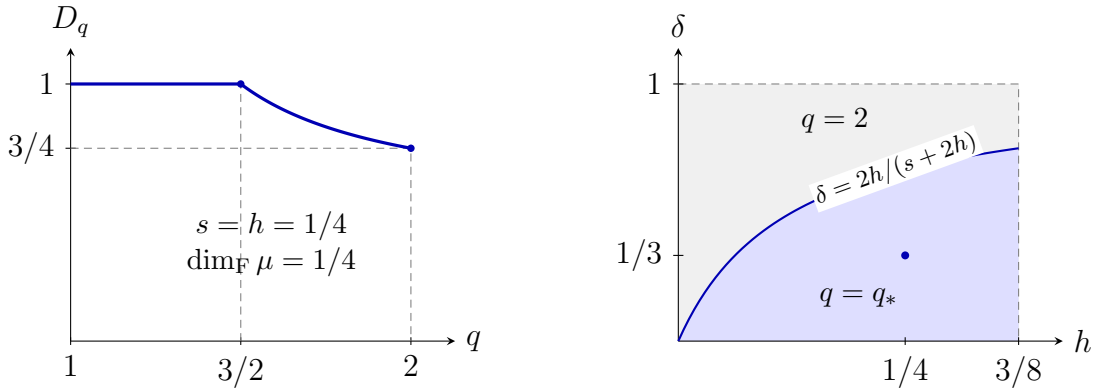
\begin{figure}[ht]
\centering
\begin{minipage}[t]{.47\textwidth}
\centering
\begin{tikzpicture}[x=4.5cm,y=3.4cm,>=stealth]
\draw[->] (0,0) -- (1.13,0) node[right] {$q$};
\draw[->] (0,0) -- (0,1.14) node[above] {$D_q$};
\foreach \x/\lab in {0/1,.5/{3/2},1/2}
  \draw (\x,.015)--(\x,-.015) node[below] {$\lab$};
\foreach \y/\lab in {.75/{3/4},1/1}
  \draw (.015,\y)--(-.015,\y) node[left] {$\lab$};
\draw[densely dashed,gray] (0,.75)--(1,.75)--(1,0);
\draw[densely dashed,gray] (.5,0)--(.5,1);
\draw[blue!70!black,very thick] (0,1)--(.5,1);
\draw[blue!70!black,very thick,domain=.5:1,samples=55]
  plot (\x,{.5+.25/\x});
\fill[blue!70!black] (.5,1) circle (1.4pt) (1,.75) circle (1.4pt);
\node[align=center,font=\small] at (.59,.38) {$s=h=1/4$\\$\dimF\mu=1/4$};
\end{tikzpicture}
\end{minipage}\hfill
\begin{minipage}[t]{.47\textwidth}
\centering
\begin{tikzpicture}[x=4.5cm,y=3.4cm,>=stealth]
\fill[gray!12] (0,0) rectangle (1,1);
\fill[blue!13] (0,0) -- plot[domain=0:1,samples=65]
  (\x,{3*\x/(1+3*\x)}) -- (1,0) -- cycle;
\draw[->] (0,0)--(1.13,0) node[right] {$h$};
\draw[->] (0,0)--(0,1.14) node[above] {$\delta$};
\draw[densely dashed,gray] (0,1)--(1,1)--(1,0);
\draw[blue!70!black,thick,domain=0:1,samples=65]
  plot (\x,{3*\x/(1+3*\x)});
\draw (.6667,.015)--(.6667,-.015) node[below] {$1/4$};
\draw (1,.015)--(1,-.015) node[below] {$3/8$};
\draw (.015,1)--(-.015,1) node[left] {$1$};
\draw (.015,.3333)--(-.015,.3333) node[left] {$1/3$};
\fill[blue!70!black] (.6667,.3333) circle (1.5pt);
\node[font=\small] at (.46,.86) {$q=2$};
\node[font=\small] at (.53,.17) {$q=q_*$};
\node[font=\scriptsize,rotate=20,fill=white,inner sep=1pt]
  at (.64,.67) {$\delta=2h/(s+2h)$};
\end{tikzpicture}
\end{minipage}
\small
\caption{Left: the exact $L^q$ dimensions in Example~\ref{ex:Lq-gain}, on the proved range $1<q\le2$. Right: the optimizer of the sufficient criterion \eqref{eq:closed-restriction-phase} for $s=1/4$, $0<h<3/8$. The point represents $(h,\delta)=(1/4,1/3)$. The two regions describe the proven estimates, not the unknown optimal restriction exponent.}
\label{fig:Lq-phase}
\end{figure}

There is also a density interpretation of the full-dimensional branch.
Fix $1<q<\infty$. If $h>0$, $\delta>0$, and $(q-1)(1-s)<qh$, then
the approximating densities satisfy
\begin{equation*}
\|f_N\|_q^q=\prod_{n\le N}R_n^{q-1}S_n(q),
\qquad \sup_N\|f_N\|_q<\infty.
\end{equation*}
Splitting the two digit classes in \eqref{eq:weighted-q-exact} gives
\[
R_n^{q-1}S_n(q)\le1+C_q\bigl(\alpha_n+
\alpha_n^q(R_n/M_n)^{q-1}\bigr).
\]
Both error sequences are summable because $\delta>0$ forces exponential growth of $R_n$ and the second exponent is strictly negative. The $L^q$ martingale theorem gives an $L^q$ density. The H\"older and Hausdorff--Young inequalities then yield the restriction estimate for $p'\ge2q'$. In the last case of \eqref{eq:closed-restriction-phase}, letting $q\uparrow q_*$ recovers its strict restriction range. This gives a direct integrability explanation of that branch.
In particular, the measure in Example~\ref{ex:Lq-gain} is absolutely continuous.

A strict improvement over the Stein--Tomas range also occurs outside this
$L^q$-density regime, as the following singular continuous example shows.
\begin{example}\label{ex:singular-Lq-gain}
Use the same digit sets and scales as in Example~\ref{ex:Lq-gain}, but prescribe
\[
\alpha_n=\begin{cases}
1,&n=k^2\text{ for some }k\ge1,\\
R_n^{-1/4},&\text{otherwise}.
\end{cases}
\]
The resulting measure is singular continuous and
\begin{equation}\label{eq:singular-weighted-dimensions}
\dimF\mu=\frac1{12},\quad
D_q(\mu)=\frac1{12}+\frac23
\min\left\{1,\frac14+\frac{q}{4(q-1)}\right\}
\quad(1<q\le2).
\end{equation}
In particular $D_{3/2}=3/4$, $D_2=7/12$, and $D_\infty\le5/12$. Thus restriction holds for every $p<18/17$, whereas the ordinary Stein--Tomas criterion cannot give $p>30/29$.

\end{example}
\begin{proof}
At the exceptional levels the one-level Fourier exponent $a_n$ from Theorem~\ref{thm:weighted} tends to $1/4$, and at the other levels it tends to $3/4$. Hence $\dimF\mu=1/12$. For a fixed $q\in(1,2]$, put
\[
b(q)=\min\left\{1,\frac14+\frac{q}{4(q-1)}\right\},\quad
c_N(q)=\frac{-\sum_{n\le N}\log S_n(q)}{(q-1)T_N}.
\]
Up to an error tending to zero, the recurrence for this average is
\[
c_N(q)=(1-\delta_N)c_{N-1}(q)+\delta_N
\begin{cases}1/4,&N=k^2,\\ b(q),&\text{otherwise}.
\end{cases}
\]
The gaps between exceptional indices tend to infinity. On an ordinary run, the recurrence contracts the difference $|c_N(q)-b(q)|$ by a factor tending to $2/3$ at each step, up to an error tending to zero. Hence the averages just before the exceptional levels tend to $b(q)$, and just after them tend to $2b(q)/3+1/12$. Between exceptional levels they return towards $b(q)$, up to a vanishing error. Thus their lower limit has precisely the value in \eqref{eq:singular-weighted-dimensions}, by Theorem~\ref{thm:weighted-Lq}. Repeating the same recurrence for the largest cylinder mass gives
$D_\infty\le(2/3)(1/2)+(1/3)(1/4)=5/12$.

At each exceptional level all digits outside $B_n$ have zero mass. The support has level-$N$ coverings whose total length is at most
$
\prod_{k^2\le N}\frac{M_{k^2}}{R_{k^2}}\rightarrow0.
$
This proves singularity; the vanishing largest digit weights imply continuity. Finally, the $q=3/2$ threshold in \eqref{eq:Lq-restriction-threshold} is
$
2\cdot3+\frac{4(1-3/4)}{1/12}=18,$
whereas the classical threshold is at least
$2+4(1-5/12)/(1/12)=30$. This proves the stated strict improvement for a singular member of the weighted family. Spectrality is not asserted for these unequal weights.
\end{proof}

\section*{Statements and Declarations}

\subsection*{Data Availability}
This article is theoretical. No research datasets were generated or analysed.

\subsection*{Competing Interests}
The author declares that he has no financial or non-financial interests
relevant to the content of this article.

\subsection*{Use of generative AI}
During preparation of this manuscript, the author used ChatGPT (OpenAI) to assist with language revision. Responsibility  for the mathematical arguments, citations, and final text rests with the author.

\end{document}